\documentclass[12pt,a4paper]{article}
\usepackage[utf8x]{inputenc}
\usepackage{mathtools}
\usepackage{graphicx}
\usepackage{color}
\usepackage{tikz-cd}
\usepackage[shortlabels]{enumitem}
\usepackage{geometry}

\RequirePackage{amsthm,amsmath,amsfonts,amssymb}
\RequirePackage[numbers]{natbib}

\usepackage{calc}

\usepackage{accents}
\newcommand{\dbtilde}[1]{\widetilde{\hspace{-1.5pt}\raisebox{0pt}[0.85\height]{$\widetilde{#1}$}}}

\usepackage[el,nf]{coelacanth}
\usepackage[T1]{fontenc}

\let\oldnormalfont\normalfont
\def\normalfont{\oldnormalfont\mdseries}

\newtheorem{theorem}{Theorem}[section]
\newtheorem{lemma}[theorem]{Lemma}
\newtheorem{proposition}[theorem]{Proposition}
\newtheorem{corollary}[theorem]{Corollary}

\theoremstyle{definition}
 
\newtheorem{remark}[theorem]{Remark}

\newcommand{\punkt}{\,\begin{picture}(-1,1)(-1,-3)\circle*{2.5}\end{picture}\;\; }

\newcommand{\T}{\mathbb{T}}
\newcommand{\N}{\mathbb{N}}

\newcommand{\C}{\mathbb{C}}
\newcommand{\R}{\mathbb{R}}
\newcommand{\E}{\mathbb{E}}
\newcommand{\pp}{\mathbb{P}}

\DeclareRobustCommand{\rchi}{{\mathpalette\irchi\relax}}
\newcommand{\irchi}[2]{\raisebox{\depth}{$#1\chi$}}

\renewcommand{\emph}[1]{\textbf{ #1}}

\title{Decoupling inequalities with exponential constants}

\author{Daniel Carando \thanks{Departamento de Matem\'atica, Facultad de Cs. Exactas y Naturales,
		Universidad de Buenos Aires and IMAS-UBA-CONICET, Int.~G\"uiraldes s/n, 1428, Buenos Aires, Argentina (dcarando@dm.uba.ar). Supported by CONICET-PIP 11220130100329CO,  ANPCyT PICT 2015-2299 and ANPCyT PICT
2018-04104
.}\and
Felipe Marceca\thanks{Departamento de Matem\'atica, Facultad de Cs. Exactas y Naturales,
	Universidad de Buenos Aires and IMAS-UBA-CONICET, Int.~G\"uiraldes s/n, 1428, Buenos Aires, Argentina (fmarceca@dm.uba.ar). Supported by a CONICET doctoral fellowship,  CONICET-PIP 11220130100329CO,  ANPCyT PICT 2015-2299 and ANPCyT PICT 2018-04104
.} \and
Pablo Sevilla-Peris\thanks{Instituto Universitario de Matem\'atica Pura y Aplicada,
Universitat Polit\`{e}cnica de Val\`encia, cmno Vera s/n, 46022,
Val\`encia, Spain (psevilla@mat.upv.es) Supported by MINECO and FEDER Project MTM2017-83262-C2-1-P}}

\begin{document}

\maketitle

\begin{abstract} Decoupling inequalities disentangle complex dependence structures of random objects so that they can be analyzed by means of standard tools from the theory of independent random variables.
We study decoupling inequalities for vector-valued homogeneous polynomials evaluated at random variables.  We focus on providing geometric conditions ensuring decoupling inequalities with good constants depending only exponentially on the degree of the polynomial. Assuming the Banach space has finite cotype we achieve this for classical decoupling inequalities that compare the polynomials with their associated multilinear operators.
Under stronger geometric assumptions on the involved Banach spaces, we also obtain decoupling inequalities between random polynomials and fully independent random sums of their coefficients.
Finally, we present decoupling inequalities where in the multilinear operator just two independent copies of the random vector are involved (one repeated $m-1$ times).
\end{abstract}

\section*{Introduction}

The decoupling principle consists in introducing enough independence to make a complex problem more manageable. More precisely, decoupling inequalities
compare objects involving heavily dependent random variables to simpler ones where the dependence structure is weaker.

In this work we present several decoupling inequalities for random homogeneous polynomials (precise definitions are given below).
Multivariate polynomials evaluated at random variables have at first glance a highly dependent structure, since each random variable appears in several monomials. Decoupling inequalities disentangle this complex structure introducing enough independence to use tools from the theory of independent random variables (see \cite{DPGi99}).

Notice that if $P:\mathbb C^n\to X$ is a vector-valued $m$-homogeneous polynomial and $M$ is its associated symmetric $m$-linear operator, we can write
\[
P(z)=M(z,\ldots,z).
\]
When $\xi$ is  a random vector  and $\xi^{(1)},\ldots,\xi^{(m)}$ are independent identically distributed (`iid' from now on) copies of  $\xi$, the random variable $M(\xi^{(1)},\ldots,\xi^{(m)})$ is, from a probabilistic point of view, a decoupled alternative to $P(\xi)$. Heuristically, the variables appearing in the monomials of $M$ are less intertwined, which leads to a weaker interdependence.

Decoupling inequalities in this context were first introduced by McConnell and Taqqu in \cite{mctaq, mctaq0} and further studied by de Acosta in \cite{acosta} and Kwapie\'n in \cite{Kw87} among others. These works established inequalities comparing the moments of $M(\xi^{(1)},\ldots,\xi^{(m)})$ to those of $P(\xi)$ with constants depending on the degree $m$ but not on the number of variables $n$ of the polynomial. The dependence on $m$ of these constants improves considerably when restricting ourselves to $\alpha$-stable random variables and are in some sense optimal for gaussian random variables (see \eqref{gauss}). Our main objective is to provide geometric conditions on the Banach space to ensure good constants (of the form $C^m$) for arbitrary symmetric random variables.

In Theorem~\ref{decoup} we give a decoupling inequality for $p$-moments of tetrahedral polynomials. This gives, for Banach spaces of non-trivial cotype, better constants than those derived from Kwapien's general result  \cite[Theorem~2]{Kw87}. We obtain Theorem~\ref{decoup} as a consequence of Theorem~\ref{compara} and Remark~\ref{remarkgauss}, which essentially show that for spaces of non-trivial cotype the random vector $\xi$  of a random polynomial can be changed without losing control of its norm. Since polynomials in gaussian random variables satisfy good classical decoupling inequalities, Theorem~\ref{decoup} follows.
The key points for Theorem~\ref{compara} are \eqref{eqdecom}, where we introduce a novel decomposition for homogeneous tetrahedral polynomials in terms of an average of multilinear operators, and the decoupling inequality presented in Proposition~\ref{decosim2}.

As it was just mentioned, the monomials in a random multilinear operator show much less dependence than those of the corresponding random polynomial. However, some dependence remains. Under  stronger geometric conditions on the involved Banach space, in Section~\ref{sectotin} we show decoupling inequalities between arbitrary polynomials and fully independent sums of their coefficients. A result in this direction was obtained in \cite{ruc} for Steinhaus random variables assuming type or cotype 2 of the Banach space $X$. Regrettably, if one needs estimates both from above and below, one must assume $X$ has type and cotype 2, which means that $X$ must be isomorphic to a Hilbert space where all these estimates hold trivially. However, the result from \cite{ruc} holds for absolute constants. Allowing for some dependence on the degree of the polynomial (constants of the form $C^m$), we can relax the geometric restrictions on the Banach space. We work with the Gaussian average property (GAP) introduced in \cite{CaNi97} and in Theorems \ref{propgap} and \ref{propgapduo} we get decoupling inequalities relating (not necessarily tetrahedral) $X$-valued random polynomials and sums of independent random variables related to the polynomials' coefficients. In particular, two-sided estimates hold for spaces including Banach lattices of non-trivial type. We also show analogous estimates for tetrahedral polynomials in symmetric random variables.

Finally, in Section~\ref{sec1dec} we study one-variable decoupling inequalities that compare $P(\xi)=M(\xi,\ldots,\xi)$ to $M(\xi',\xi,\ldots,\xi)$, where we replace $\xi$ with an iid copy $\xi'$ in only one entry. In some sense, one-variable decoupling can be seen as an averaged version of the so called Markov type inequalities for homogeneous polynomials studied by Harris in \cite{harris}. In this context, gaussian variables also satisfy optimal one-variable decoupling inequalities. In Theorem~\ref{1decoz} we show that for $K$-convex Banach spaces, Steinhaus variables behave in the same way.

Let us point out that the Bohr radius and the Bohnenblust-Hille inequality were recently studied in the context of functions on the Boolean cube $\{-1,1\}^n$ in \cite{DeMaPe18,DeMaPe19}, showing an intriguing link between the Bohnenblust-Hille inequality and quantum query complexity. Since functions on the boolean cube can be thought of as random functions on Rademacher variables, we believe that this decoupling approach could have further applications in this setting (see \cite{1deco}).

%
%

\section{Preliminaries}

A \emph{polynomial} of $n$ variables with values in some Banach space $X$ is a function $P:\C^n\to X$ given by a finite sum
\[
P(z) = \sum_{\alpha \in \Lambda \subseteq \mathbb{N}_{0}^{n}} x_{\alpha} z_1^{\alpha_1}\ldots z_n^{\alpha_n},
\]
where $x_{\alpha} \in X$ for every $\alpha\in \Lambda$. We purposely write vector times scalar rather than scalar times vector to emphasize the polynomial structure. Also, we write $z^\alpha= z_1^{\alpha_1}\ldots z_n^{\alpha_n}$ for short.

The degree of a polynomial is the maximum of $\vert \alpha \vert = \alpha_{1} + \cdots + \alpha_{n}$ over every $\alpha \in \Lambda$ such that $x_\alpha\neq 0$. We say that $P$ is \emph{$m$-homogeneous} if $\vert \alpha \vert = m$ for every $\alpha \in \Lambda$ with $x_\alpha\neq 0$. Whenever the number of variables $n$ is implicit we write $m$-homogeneous polynomials as
\[P(z) = \sum_{|\alpha|=m} x_{\alpha} z^{\alpha},\]
allowing for some coefficients $x_\alpha$ to vanish.

A polynomial is said to be \emph{tetrahedral} whenever each variable appears with exponent at most 1. In other words, monomials of a tetrahedral polynomial of $n$ variables can be indexed by $\alpha\in\{0,1\}^n$ and therefore we write
\[
P(z) = \sum_{\alpha \in  \{0,1\}^{n}} x_{\alpha} z^{\alpha}.
\]

Next, we introduce the Walsh notation for tetrahedral polynomials that we frequently use. We can identify $\alpha \in \{0,1\}^n$ with the set $A\subseteq\{1,\ldots,n\}$ that indicates which elements $1\le k \le n$ satisfy that $\alpha_k=1$. More precisely we have that $\alpha=\rchi_A$. This one-to-one correspondence allows us to index tetrahedral polynomials using subsets of $\{1,\ldots,n\}$. By a slight abuse of notation, writing $x_A$ for $x_\alpha$ whenever $\alpha=\rchi_A$ (and denoting $[n]=\{1,\ldots,n\}$), we get
\[
P(z)=\sum_{\alpha \in  \{0,1\}^{n}} x_{\alpha} z^{\alpha}=\sum_{\alpha \in  \{0,1\}^{n}} x_{\alpha} \prod_{\substack{1\le k \le n \\ \alpha_k=1}}z_k
=\sum_{A\subseteq [n]} x_A \prod_{k\in A} z_k.
\]
Letting $z_A=\prod_{k\in A} z_k$ leads to the new notation for tetrahedral polynomials:
\[P(z)=\sum_{A\subseteq [n]} x_A z_A.\]
For $m$-homogeneous tetrahedral polynomials we write
\[P(z)=\sum_{|A|=m} x_A z_A,\]
where $|A|$ stands for the cardinal of $A$.

For an $m$-homogeneous polynomial $P:\C^n\to X$ there exists unique symmetric $m$-linear operator $M:(\C^n)^m\to X$ such that $M(z,\ldots,z)= P(z)$ for every $z\in\C$ (see e.g. \cite[Section~1.1]{LibroDi99}), we call it \emph{the symmetric $m$-linear operator associated to} $P$. The operator $M$ can be retrieved from the polynomial $P$ through the polarization formula (see \cite[Corollary~1.6]{LibroDi99}): for every $z^{(1)},\ldots,z^{(m)}\in\C$, we have
\[M\left(z^{(1)},\ldots,z^{(m)}\right)=\frac{1}{m!}\E_\varepsilon \left[ \varepsilon_1\ldots\varepsilon_m P\left(\varepsilon_1 z^{(1)}+ \ldots + \varepsilon_m z^{(m)}\right)\right],\]
where $\varepsilon_1, \ldots ,\varepsilon_m$ are \emph{independent Rademacher variables} (random variables that take the values $\pm1$ with probability $1/2$). This identity allows us to relate the norm of a homogeneous polynomial with the norm of its associated multilinear operator. A straightforward argument shows that for every norm $\|\punkt \hspace{-0.3cm}\|$ on $\C^n$ and every $m$-homogeneous polynomial $P:\C^n\to X$ we have
\begin{align*}
\sup_{\|z\|\leq 1} \| P ( z ) \|_X\le \sup_{\left\|z^{(k)} \right\|\leq 1} \left\| M ( z^{(1)},\ldots , z^{(m)} ) \right\|_X \leq e^m \sup_{\|z\|\leq 1} \| P ( z ) \|_X.
\end{align*}

\begin{remark}\label{alam}
At first glance the bound $e^m$ from the previous proposition may seem quite big. However, estimates of the form $C^m$ appear naturally while working with $m$-homogeneous polynomials. Moreover, these estimates can be compensated by contracting the polynomials since for an $m$-homogeneous polynomial $P$ we have that $P(rz)=r^mP(z)$. Having this type of control is usually sufficient to carry results from the polynomial setting to vector-valued holomorphic functions or Fourier and Dirichlet series (see \cite[Chapters 23--26]{DeSe19_libro}).
Intuitively, contracting a function by some factor $r$ shrinks its homogeneous parts by a factor of $r^m$ leaving room for constants $C^m$ to appear. As a naive example of this phenomenon, notice that if $(a_m)_{m\in\N}\subseteq \C$ and $|a_m|\leq C^m$ for every $m\in\N$ then the series
\[\sum_{m\in\N}a_m z^m,\]
converges in a neighbourhood of 0. A growth of the coefficients greater than $C^m$ such as $m^m$ would have meant that the series diverges at every $z\neq 0$.
With this in mind we usually look for  $C^m$-type bounds. We  write $a\simeq_{C^m}b$ whenever $C^{-m}a\le b \le C^m a$ and say $a$ and $b$ are equivalent up to a constant $C^m$. If such an equivalence does not hold we will specify which inequality fails, if not both.
\end{remark}


We work with polynomials on random variables. Let's write  $\T=\{z\in \C : \, |z|=1\}$ for the torus.
From a probabilistic point of view, polynomials restricted to $\T^n$  can be interpreted as being evaluated at independent \emph{Steinhaus variables} (random variables uniformly distributed in the torus), so we call them \emph{Steinhaus polynomials}. For random vectors whose coordinates are independent Steinhaus variables we use the notation $w=(w_1,\ldots,w_n)$.

The following polynomial Kahane-Khinchin inequality was established in \cite[Theorem~9]{Ba02} for the scalar case and in \cite[Lemma~1.3]{CaDeSe16} for the general case (see also \cite[Theorems~8.10 and~25.9]{DeSe19_libro}). A remarkable characterization of random variables satisfying a similar result was obtained in \cite[Theorem~2.2]{KwSz}.

\begin{theorem}
\label{poiss}
 For every Banach space $X$, every $1\le p \le q <\infty$ and every polynomial $P:\C^n\to X$ we have
 \[
	\Big(\E\Big\Vert  P\Big(\sqrt{\frac{p}{q}} w\Big)  \Big\Vert^q\Big)^{1/q}
	\le
	  (\E\Vert  P (w) \Vert^p)^{1/p}.
\]
\end{theorem}

A Walsh polynomial is a random variable $P(\varepsilon)$ where $P:\C^n\to X$ is a polynomial and $\varepsilon_1,\ldots,\varepsilon_n$ are independent Rademacher variables. Since for a Rademacher variable $\varepsilon_0$ we have that $\varepsilon_0^2=1$, Walsh polynomials can always be written as $P(\varepsilon)$ where $P$ is a tetrahedral polynomial. Therefore, we can use the tetrahedral notation introduced at the beginning of this section and write
\[P(\varepsilon)=\sum_{A\subseteq [n]}  x_{A} \varepsilon_{A}.\]

We recall now the geometric notions of type and cotype. A Banach space $X$ is said to have \emph{cotype} $2 \leq q < \infty$ if there is a constant $C\geq 1$ such that for every $n \in \mathbb{N}$ and every $x_{1}, \ldots , x_{n} \in X$ we have
\begin{equation} \label{defcotipo}
\Big( \sum_{i=1}^{n} \Vert x_{i} \Vert^{q} \Big)^{1/q}
\leq C  \Big( \mathbb{E} \Big\Vert \sum_{i=1}^{n} x_{i} \varepsilon_{i} \Big\Vert^{q} \Big)^{1/q} \,,
\end{equation}
and \emph{type} $1 \leq p \leq 2$ if there is a constant $C\geq 1$ such that for every $N \in \mathbb{N}$ and every $x_{1}, \ldots , x_{N} \in X$ we have
\begin{equation*}
\bigg( \E \Big\Vert \sum_{n=1}^{N} x_{n} \varepsilon_{n} \Big\Vert^{p}  \bigg)^{1/p} \leq C \Big( \sum_{n=1}^{N} \Vert x_{n} \Vert^{p} \Big)^{1/p} \,.
\end{equation*}
We say that $X$ has non-trivial cotype if it has cotype $q$ for some $2 \leq q < \infty$. If this is not the case, then the space is said to have trivial cotype. Analogously, $X$ has non-trivial type if has type $p$ for some $1<p\le 2$. We write $\cot(X)$ for the infimum over all $q$ such that $X$ has cotype $q$.

The Rademacher variables in the previous definitions can be replaced by Steinhaus random variables (changing the constant). This is a consequence of the the well known contraction principle (see \cite[Theorem~12.2]{DiJaTo95} or \cite[Corollary~4]{seig}), which will also be helpful for us here.

\begin{theorem}[Contraction principle]\label{contr}
Let $X$ be a Banach space and fix $1\le p <\infty$. For every $\lambda\in \R^n$ and every choice of vectors $x_{1}, \ldots , x_{n} \in X$ we have
\[
\Big(\E\Big\|\sum_{j=1}^n  \varepsilon_j \lambda_j x_j \Big\|^p \Big)^{1/p}
\le \|\lambda\|_\infty \Big(\E\Big\|\sum_{j=1}^n  \varepsilon_j x_j \Big\|^p \Big)^{1/p}.
\]
Similarly, if $\lambda\in \C^n$ (and $X$ is a complex Banach space) we get
\[
\Big(\E\Big\|\sum_{j=1}^n  \varepsilon_j \lambda_j x_j \Big\|^p \Big)^{1/p}
\le \frac{\pi}{2}\|\lambda\|_\infty \Big(\E\Big\|\sum_{j=1}^n  \varepsilon_j x_j \Big\|^p \Big)^{1/p}.
\]
\end{theorem}

In particular, for every (complex) Banach space $X$, every $1\le p <\infty$ and every choice of vectors $x_{1}, \ldots , x_{n} \in X$ we have
\begin{equation}\label{zvse}
  \frac{2}{\pi} \Big(\E\Big\|\sum_{j=1}^n  \varepsilon_j x_j \Big\|^p \Big)^{1/p}
\le \Big(\E\Big\|\sum_{j=1}^n  w_j x_j \Big\|^p \Big)^{1/p}
\le \frac{\pi}{2} \Big(\E\Big\|\sum_{j=1}^n  \varepsilon_j x_j \Big\|^p \Big)^{1/p},
\end{equation}
so Rademacher and Steinhaus random sums have comparable moments.

We end this section setting some notation that will be used throughout.
As we have already mentioned, we write $[n] = \{ 1, \ldots, n  \}$. We also write
\begin{equation}\label{eq-subconjuntos}
\mathcal{P}_m[n]=\{A\subseteq [n] : \ |A|=m\}.
\end{equation}
Finally, given two vectors $x,y$ we write $xy$ for the pointwise product.

\section{General decoupling}\label{secdeco}

Our starting point is a remarkable result due to Kwapie\'n \cite[Theorem~2  and Remark~1]{Kw87} (see also \cite[Theorem~6.4.1 and Remark~6.4.1]{KwWo92}), originally stated only for real polynomials, but whose proof is
easily adapted to the complex case. It shows that if $P:\C^n\rightarrow X$ is an $m$-homogeneous tetrahedral polynomial with associated symmetric $m$-linear operator $M$ and $\Phi:X\rightarrow \R_{\geq0}$ is a convex function such that $\Phi(x)=\Phi(- x)$ for every $x\in X$, then
\begin{equation}\label{decosim}
	\E\Phi \big( \tfrac{1}{m^{m}} P(\xi) \big)
	\leq \E\Phi(M(\xi^{(1)},\ldots,\xi^{(m)}))
	\leq \E\Phi \big( \tfrac{m^{m}}{m!}P(\xi) \big) \,,
\end{equation}
for every vector $\xi=(\xi_1,\ldots,\xi_n)$ of independent symmetric entries, where  $\xi^{(1)},\ldots,\xi^{(m)}$ are iid copies of $\xi$. For \emph{gaussian random vectors} (which we always assume to have standard complex gaussian coordinates, and denote
$\gamma,\gamma^{(1)},\ldots,\gamma^{(m)}$), the inequalities can be improved to
\begin{equation}\label{decogauss}
	\E\Phi \big(\tfrac{1}{m^{m/2}} P(\gamma) \big)
	\leq \E\Phi(M(\gamma^{(1)},\ldots,\gamma^{(m)}))
	\leq \E\Phi\big( \tfrac{m^{m/2}}{m!}P(\gamma) \big).
\end{equation}

We are particularly interested on the $p$-norm of random polynomials, that is taking $\Phi = \Vert \punkt \Vert^{p}$ for some $1 \leq p < \infty$. A straightforward computation using Stirling's formula yields
\[
\frac{m^{m/2}}{m!} \leq  \frac{e^{m}}{m^{m/2}} \,.
\]
Using this and rearranging~\eqref{decogauss} in order to put $P$ in a central role we get
\[
\frac{m^{m/2}}{e^m} \big( \mathbb{E} \Vert M(\gamma^{(1)},\ldots,\gamma^{(m)})\Vert^{p} \big)^{1/p}\!
\leq \big( \mathbb{E} \Vert P(\gamma)\Vert^{p} \big)^{1/p}\!
\leq  m^{m/2} \big( \mathbb{E} \Vert M(\gamma^{(1)},\ldots,\gamma^{(m)})\Vert^{p} \big)^{1/p}
\]
or, to put it in other terms,
\begin{equation}  \label{gauss}
 (\E\|P(\gamma)\|^p)^{1/p}
 \simeq_{C^m} m^{m/2} (\E\|M(\gamma^{(1)},\ldots,\gamma^{(m)})\|^p)^{1/p}.
\end{equation}
In other words, the $p$-norm of a gaussian polynomial can be estimated up to a constant $C^m$ computing the $p$-norm of its associated $m$-linear operator.
As mentioned in Remark~\ref{alam}, circumstances where constants $C^m$ are tolerated are commonplace when working with polynomials of degree $m$. However, starting with~\eqref{decosim} and doing the same for arbitrary symmetric random vectors $\xi$ we deduce
\[
\frac{1}{e^m} \big(\E\|M(\xi^{(1)},\ldots,\xi^{(m)})\|^p \big)^{1/p}
\leq \big(\E\|P(\xi)\|^p \big)^{1/p}
\leq m^{m} \big(\E\|M(\xi^{(1)},\ldots,\xi^{(m)})\|^p \big)^{1/p}
\]
and a gap of order $m^m$ remains, which can be too big for some applications.\\

Our aim now is to show that, under  not too demanding assumptions on the space and on the random variables, we can get a `good' estimation (in the sense that constants like $C^{m}$ appear) as in~\eqref{gauss}. This is the main result of this section.

\begin{theorem}\label{decoup}
  	Let $X$ be a Banach space of finite cotype, let $\xi_0$ be a non-trivial symmetric random variable with finite $s$-norm for some $s> \cot(X)$ and fix $1\leq p <s$. There is a constant $C\ge 1$ such that for every $m$-homogeneous tetrahedral polynomial $P:\C^n\rightarrow X$ we have
  	\begin{equation}  	\label{decoupsim}
  	(\E\|P(\xi)\|^p)^{1/p}
  	\simeq_{C^m} m^{m/2} (\E\|M(\xi^{(1)},\ldots,\xi^{(m)})\|^p)^{1/p},
  	\end{equation}
  	where $\xi,\xi^{(1)},\ldots,\xi^{(m)}$ are independent random vectors whose coordinates are iid copies of~$\xi_0$.
\end{theorem}

Theorem~\ref{decoup} follows from the following polynomial version of \cite[Proposition~3.2]{pisier} that compares the $p$-norms of a tetrahedral polynomial evaluated in different random vectors. We show that under certain mild conditions on the space and on the
random vectors, these (the random  vectors) are essentially interchangeable.

\begin{theorem}\label{compara}
 Let $X$ be a Banach space, $\xi_{0}$ a non-trivial symmetric random variable and $\xi$ a random vector of iid copies of $\xi$.
\begin{enumerate}[(a)]
\item \label{landa1} There is a constant $C\ge 1$ such that
\begin{equation} \label{walshchica}
\big(\mathbb{E}\|P(w)\|^p\big)^{1/p}\leq C^m \big(\E\|P(\xi)\|^p \big)^{1/p}
\end{equation}
for every tetrahedral polynomial $P:\C^n\rightarrow X$ of degree $m$ and every $1 \leq p < \infty$.

\item \label{landa2} If $X$ has non-trivial cotype,  $\xi_0$ has finite $s$-norm for some $s> \cot(X)$ and $1\leq p <s$, then there is a constant $C\ge 1$ such that
\begin{equation} \label{propcomparacion}
\big(\mathbb{E}\|P(\xi)\|^p \big)^{1/p} \leq C^m \big(\E\|P(w)\|^p \big)^{1/p}
\end{equation}
for every tetrahedral polynomial $P:\C^n\rightarrow X$ of degree $m$.
\end{enumerate}
 \end{theorem}

\begin{remark}\label{remarkgauss}
The linear result \cite[Proposition~3.2]{pisier} is stated for Rademacher rather than Steinhaus variables. However, these are interchangeable by virtue of \cite[Lemma~4.2]{CaMaSe}  (see Lemma~\ref{lemma1} below). We can also replace  the Steinhaus variables with gaussian random variables for spaces with finite cotype using the theorem twice.
Indeed, let $X$ be a Banach space of finite cotype, let $P:\C^n\rightarrow X$ of degree $m$ be a  tetrahedral polynomial and fix $1 \leq p < \infty$.
Since gaussian random variables have finite $s$-norm for every $s$, using Theorem~\ref{compara} twice gives
\[
\big( \mathbb{E}\|P(\gamma)\|^p \big)^{1/p}
 \leq C_{1}^{m} \big( \E\|P(w)\|^p \big)^{1/p}
\leq (C_{1} C_{2})^{m} \big( \E\|P(\xi)\|^p \big)^{1/p}
\]
for every $\xi$  consisting of iid copies of some symmetric random variable regardless that it has finite $s$-norm or not
(note that in the first inequality we are using that $X$ has finite cotype, so that we can apply~\eqref{propcomparacion} to gaussian variables, while in the second one we just use~\eqref{walshchica} for the variables $\xi$).

If $\xi_{0}$ has finite $s$-norm for some $s > \max(\cot(X),p)$ we can use the same argument to obtain the converse inequality.
\end{remark}

Before we go any further, let us show how Theorem~\ref{decoup} follows from all this.

\begin{proof}[Proof of Theorem~\ref{decoup}]
Note that if $P$ is a tetrahedral $m$-homogeneous polynomial of $n$ variables, then its associated $m$-linear operator $M$ can also be regarded as an $m$-homogeneous tetrahedral polynomial of $nm$ variables. Then the result follows from the previous remark, since it allows us to reduce~\eqref{decoupsim} to~\eqref{gauss} by replacing $P(\xi)$ with $P(\gamma
)$, as well as $M(\xi^{(1)},\ldots,\xi^{(m)})$ with $M(\gamma^{(1)},\ldots,\gamma^{(m)})$, where $\gamma,\gamma^{(1)},\ldots,\gamma^{(m)}$ are independent gaussian random vectors.
\end{proof}

Next we discuss the necessity of the hypotheses in Theorems~\ref{decoup} and~\ref{compara}.
A simple computation shows that the hypothesis of $P$ being tetrahedral is needed in both theorems. Just taking $P(z) = z^{m}$, we have  $M(z^{(1)},\ldots,z^{(m)})=z^{(1)}\cdots z^{(m)}$ and
\[
\E|M(w^{(1)},\ldots,w^{(m)})|^p
  	=\E|w^{(1)} \cdots w^{(m)}|^p=1=\E|P(w)|^p.
\]
This shows that the inequality
\[
m^{m/2} (\E\|M(w^{(1)},\ldots,w^{(m)})\|^p)^{1/p} \leq C^{m} 	(\E\|P(w)\|^p)^{1/p}
\]
does not hold in general. Regarding Theorem~\ref{compara}, notice that
\[
\E|P(w)|^p=\E|w^m|^p=1.
\]
On the other hand, since $2|\gamma|^2$ has a chi-squared distribution with two degrees of freedom, a straightforward computation shows that for every $q>0$ we have
\[\E|\gamma|^q=\Gamma\Big(\frac{q}{2}+1\Big).\]
Using Stirling's formula we get
\[\E|P(\gamma)|^p=\E|\gamma|^{pm}
=\Gamma\Big(\frac{pm}{2}+1\Big)
\simeq_{C^m} m^{pm/2},\]
  	so~\eqref{propcomparacion} also fails.

The hypothesis of $X$ having non-trivial cotype is also necessary in Theorem~\ref{compara}. In \cite[page~253]{contrac} it is shown that~\eqref{propcomparacion} may fail for spaces with trivial cotype even for $m=1$. The same is true for Theorem~\ref{decoup}, but this requires some extra work.
A careful look at the proof of \cite[Theorem~2]{Kw87} shows that for every $m$-homogeneous tetrahedral polynomial $P$ and every symmetric convex function $\Phi:X\rightarrow \R_{\geq0}$ we have
\[
\E\Phi\Big(m^{-m}P\Big(\sum_{l=1}^m \xi^{(l)}\Big)\Big)
	\leq \E\Phi(M(\xi^{(1)},\ldots,\xi^{(m)})).
\]
In particular, letting $\Phi=\|\punkt \|^p$,
we get
\begin{equation} \label{ecdecosuma}
\Big(\E\Big\|P\Big(\sum_{l=1}^m \xi^{(l)}\Big)\Big\|^p\Big)^{1/p}
\leq m^m \big(\E\|M(\xi^{(1)},\ldots,\xi^{(m)})\|^p \big)^{1/p}.
\end{equation}
This allows to show that for Rademacher variables, Theorem~\ref{decoup} fails for \textit{ every} space $X$
 of trivial cotype.

\begin{remark} \label{remate} Let $X$ be a Banach space with trivial cotype and suppose that we can find some $C \geq 1$ so that
\begin{equation} \label{bluemoon}
m^{m/2} \big( \E\|M(\varepsilon^{(1)},\ldots,\varepsilon^{(m)})\|_X^p \big)^{1/p}
\leq C^{m} \big(\E \|P ( \varepsilon)  \Vert_{X}^{p} \big)^{1/p}
\end{equation}
for every tetrahedral $m$-homogeneous polynomial $P : \mathbb{C}^{n} \to X$.
With the notation from \eqref{eq-subconjuntos}), for $m,n\in\N$ let $\ell_\infty(\mathcal{P}_m[n])$ be the normed space $(\C^{\binom{n}{m}}, \|\punkt \hspace{-0.1cm}\|_\infty)$ where coordinates are indexed by the sets $A\in \mathcal{P}_m[n]$ rather than natural numbers.
Recall that a Banach space has trivial type if and only if the finite dimensional $\ell_\infty^k$ spaces can be included in $X$ for every $k\in\N$ with uniform distortion (see for example \cite[Theorem~14.1]{DiJaTo95}).
So, there is a constant $\widetilde{C}\geq 1$ such that
\begin{equation} \label{marcello}
m^{m/2} \big( \E\|M(\varepsilon^{(1)},\ldots,\varepsilon^{(m)})\|_{\ell_\infty(\mathcal{P}_m[n])}^p \big)^{1/p}
\leq \widetilde{C} C^{m} \big(\E \|P ( \varepsilon)  \Vert_{\ell_\infty(\mathcal{P}_m[n])}^{p} \big)^{1/p}
\end{equation}
for every $m,n\in\N$ and every tetrahedral $m$-homogenous polynomial $P : \mathbb{C}^{n} \to \ell_\infty(\mathcal{P}_m[n])$.\\
Denote the canonical basis of $\ell_\infty(\mathcal{P}_m[n])$ by $\{e_A\}_{|A|=m}$ and consider the $m$-homogeneous polynomial $P:\mathbb{C}^n\rightarrow \ell_\infty(\mathcal{P}_m[n])$ given by
\[
P(z)=\sum_{|A|=m} e_A z_A.
\]
This simply allocates each monomial in a separate coordinate.
Notice that for every $\varepsilon\in\{-1,1\}^n$ we have
\[
\|P(\varepsilon)\|_{\ell_\infty(\mathcal{P}_m[n])}=\sup_{|A|=m} |\varepsilon_A|=1 \,,
\]
So that
\[
\big( \E\|P(\varepsilon)\|^p \big)^{1/p}=1 \,.
\]
We estimate the norm of the $m$-linear form through~\eqref{ecdecosuma}, taking a sum of $m$ independent copies of $\varepsilon$. Observe that, for a fixed set $A$, the product $\prod_{i\in A}\Big|\sum_{l=1}^m \varepsilon_i^{(l)}\Big|$ is the biggest
possible for $\varepsilon^{(l)} \in \{-1,1\}^{n}$ if $\varepsilon_i^{(l)} =1$ for every $ i \in A$ and  $l =1, \ldots ,m$. Therefore
\[
\Big(\E\Big\|P\Big(\sum_{l=1}^m \varepsilon^{(l)}\Big)\Big\|^p\Big)^{1/p}
=\Big(\E\sup_{|A|=m}\prod_{i\in A}\Big|\sum_{l=1}^m \varepsilon_i^{(l)}\Big|^p\Big)^{1/p}
\leq m^{m} \,.
\]
We obtain a lower estimate  by an infinite-monkey-theorem type of argument: arrange the Rademacher variables in an $n\times m$ matrix $(\varepsilon^{(l)}_i)_{i,l}$. Now if there are (at least) $m$ rows where every entry is 1, we can choose $A$ to index those rows to get
\[
\sup_{|A|=m}\prod_{i\in A}\Big|\sum_{l=1}^m \varepsilon_i^{(l)}\Big|=m^m.
\]
Letting $n$ tend to infinity, the probability of finding $m$ rows of ones tends to 1.
Explicitly, the probability that a given row has only ones is $2^{-m}$, so the number of rows of ones follows a binomial distribution $\text{Bi}(n,2^{-m})$. Since the probability of having a fixed number of successes ($m$ successes in our case) tends to 1 as $n$ goes to infinity, for a sufficiently large $n$ we have
\[
\mathbb{P}\Big(\sup_{|A|=m}\prod_{i\in A}\Big|\sum_{l=1}^m \varepsilon_i^{(l)}\Big|=m^m\Big)\geq \frac{1}{2}
 \,.
\]
So, from Chebyshev's inequality we get
\[
\frac{1}{2} m^{m} \leq \Big(\E\sup_{|A|=m}\prod_{i\in A}\Big|\sum_{l=1}^m \varepsilon_i^{(l)}\Big|^p\Big)^{1/p}
=\Big(\E\Big\|P\Big(\sum_{l=1}^m \varepsilon^{(l)}\Big)\Big\|^p\Big)^{1/p}  \,.
\]
Finally, if~\eqref{marcello} holds, using all these and~\eqref{ecdecosuma} we deduce
\begin{align*}
\frac{1}{2} m^{m/2}  \leq  m^{-m/2} \Big(\E\Big\|P\Big(\sum_{l=1}^m \varepsilon^{(l)}\Big)\Big\|^p\Big)^{1/p}
\leq  m^{m/2} \big( \E\|M(\varepsilon^{(1)},\ldots,\varepsilon^{(m)})\|^p \big)^{1/p}
\leq \widetilde{C} C^{m} \,,
\end{align*}
and this leads to a contradiction, showing that there is no $C\geq 1$ so that \eqref{bluemoon} holds.
\end{remark}

We proceed now with the proof of Theorem~\ref{compara}. It requires to establish first what we call decoupling on partitions (a particular sort of decoupling inequalities).

\subsection{Decoupling in partitions}

One of the  main ideas for the proof of Theorem~\ref{compara} is to device an alternative  decoupling method, associating to each $m$-homogeneous tetrahedral polynomial a family of $m$-linear operators. This is inspired by a combinatorial identity proved in \cite{RzWo_19} which is presented in Lemma~\ref{lemidcomb2}.\\
Given any  $m$-homogeneous tetrahedral polynomial of $n$-variables $P(z)=\sum x_A z_A$, without loss of generality (making $n$ bigger if necessary) we may assume that $n=km$ for some $k\in\N$.
For each  ordered partition $\pi=(B_1,\ldots,B_m)$ of $[n]$ in $m$ sets  of $k$ elements each, we define the following $m$-linear mapping:
\begin{equation}
\label{eqlpi}
L_\pi(z^{(1)},\ldots,z^{(m)})=\sum_{i_1\in B_1}\ldots \sum_{i_m\in B_m} x_{\{i_1,\ldots,i_m\}} z^{(1)}_{i_1}\ldots z^{(m)}_{i_m}.
\end{equation}
Observe that $L_\pi(z,\ldots,z)$ can be obtained from $P(z)$ by keeping only the monomials whose index $A$ has exactly one element in each set $B_l$ of the partition. Let us make this statement more precise. Consider the linear transformation $T_\pi:\C^m\rightarrow\C^n$ given by
\[
T_\pi(e_l)=\sum_{j\in B_l}e_j.
\]
Take some Rademacher random vector $\varepsilon=(\varepsilon_1,\ldots,\varepsilon_m)$ and note that, for each fixed $A$, the expectation
\[
\E_\varepsilon\Big[\varepsilon_1\cdots\varepsilon_m\prod_{i\in A} T_\pi(\varepsilon)_i\Big]
\]
is $1$ if $A$ has exactly one element in each set  of the partition and $0$ otherwise. Then
\[
\E_\varepsilon[\varepsilon_1\cdots\varepsilon_m P(T_\pi(\varepsilon) z)]= \sum_{|A|=m} x_A \E_\varepsilon\Big[\varepsilon_1\cdots\varepsilon_m\prod_{i\in A} T_\pi(\varepsilon)_i\Big] z_A \,,
\]
and (recall that $T_\pi(\varepsilon) z$ denotes the pointwise product)
\begin{equation}\label{lpi}
L_\pi(z,\ldots,z)=\E_\varepsilon[\varepsilon_1\cdots\varepsilon_m P(T_\pi(\varepsilon) z)] \,.
\end{equation}

We want to see now how can we recover $P(z)$ by using the $L_{\pi}$s. Loosely speaking, if we sum over all possible partitions, eventually all monomials in $P(z)$ appear in the sum, and they do it the
same amount of times. Let us expose this more systematically. Let
\begin{equation}
\label{party}
\Pi_{k,m}=\Big\{\pi=(B_1,\ldots,B_m)
\in \mathcal{P}_k([n])^m : \ \bigcup_{l=1}^m B_l =[n]
\Big\}
\end{equation}
be the familiy of all ordered partitions $\pi$ of $[n]$ in $m$ sets of $k$-elements (note that the disjointness of  the sets $B_l$ is automatic,  since they have $k$ elements and $n=km$). A symmetry
argument shows that there is some $N(k,m)\in\N$ so that
\[
\sum_{\pi\in \Pi_{k,m}} L_\pi (z,\ldots,z)=N(k,m)P(z),
\]
since each monomial appears the same number $N(k,m)$ of times.  So, the polynomial $P$ can be written as (almost) an average of this family of multilinear operators $L_\pi$ evaluated at $(z,\ldots,z)$.
The key point for us now is to estimate the growth of this number $N(k,m)$, showing that it is (up to a constant $C^{m}$) like $|\Pi_{k,m}|$ (see~\eqref{eqdecom}). This relies in the following
combinatorial  equality.

\begin{lemma}\label{lemidcomb2}
Let $V$ be a vector space and take $n=km$ where $k,m \in \N$. Given a family $\{ v_{A} \colon A \subseteq [n], \, \vert A \vert =m \}\subseteq V$ we have
\begin{equation}
\label{idcomb2}
\sum_{|A|=m} v_A = \frac{1}{k^m}\binom{km}{m} \frac{1}{|\Pi_{k,m}|} \sum_{\pi\in \Pi_{k,m}} \sum_{i_1\in B_1}\ldots \sum_{i_m\in B_m} v_{\{i_1,\ldots,i_m\}}.
\end{equation}
\end{lemma}

Before we proceed to the proof, let us note that, just taking $v_A=x_A z_A$ in the previous combinatorial identity  and using~\eqref{eqlpi} we immediately get
\begin{equation}
\label{eqdecom}
P(z)
=\frac{1}{k^m}\binom{km}{m} \frac{1}{|\Pi_{k,m}|} \sum_{\pi\in \Pi_{k,m}} L_\pi (z,\ldots,z).
\end{equation}

\begin{proof}[Proof of Lemma~\ref{lemidcomb2}]
First notice that
\begin{align}
\label{idcomb22}
\sum_{\pi\in \Pi_{k,m}} &\sum_{i_1\in B_1}\ldots \sum_{i_m\in B_m} v_{\{i_1,\ldots,i_m\}}
=\sum_{|A|=m} \Big( \sum_{\substack{\pi\in \Pi_{k,m}\\|A\cap B_l|=1, \ \forall l}}1\Big) v_A \notag
\\ &=\sum_{|A|=m}
\underbrace{m\vphantom{\sum_{|A|=m}}}_{\substack{\text{choose} \\ A\cap B_1}}
\underbrace{\binom{(k-1)m}{k-1}\vphantom{\sum_{|A|=m}}}
_{\substack{\text{choose} \\ A^c\cap B_1}}
\underbrace{(m-1)\vphantom{\sum_{|A|=m}}}_{\substack{\text{choose} \\ A\cap B_2}}
\underbrace{\binom{(k-1)(m-1)}{k-1}\vphantom{\sum_{|A|=m}}}
_{\substack{\text{choose} \\ A^c\cap B_2}}
\ldots 1 \binom{k-1}{k-1} v_A \notag
\\ &=m! \prod_{l=1}^m \binom{(k-1)l}{k-1} \sum_{|A|=m} v_A.
\end{align}
On the other hand, we have
\begin{align*}
|\Pi_{k,m}|&=\prod_{l=1}^m \binom{kl}{k}
=\frac{1}{k^m}\prod_{l=1}^m \binom{(k-1)l}{k-1} \frac{(k-1)(l-1)!}{(k-1)l!}\frac {kl!}{k(l-1)!}
\\ &=\frac{1}{k^m} \prod_{l=1}^m\frac{(k-1)(l-1)!}{(k-1)l!}
\prod_{l=1}^m\frac {kl!}{k(l-1)!}
\prod_{l=1}^m \binom{(k-1)l}{k-1}
\\ &=\frac{1}{k^m}\frac{km!}{(k-1)m!}\prod_{l=1}^m \binom{(k-1)l}{k-1}
=\frac{1}{k^m}\binom{km}{m} m!\prod_{l=1}^m \binom{(k-1)l}{k-1}.
\end{align*}
Joining this with~\eqref{idcomb22} gives the conclusion.
\end{proof}

Let us note that, by definition, we have $\binom{km}{m} \geq k^{m}$ for every $k$ and $m$. On the other hand, using Stirling's formula yields
\[
\binom{km}{m}\leq \frac{e}{2\pi \sqrt{m}} \sqrt{\frac{k}{k-1}} k^{m} \Big(\frac{k}{k-1}\Big)^{(k-1)m}
\leq \frac{e}{2\pi \sqrt{m}} \sqrt{2} k^{m} e^{m}\leq e^{m} k^{m},
\]
for $k \geq 2$ (the inequality holds trivially for $k=1$). Then,
\begin{equation} \label{dalla}
1\le \frac{1}{k^m}\binom{km}{m}\leq e^m
\end{equation}
for every $k$ and $m$. With all this at hand we can give the following decoupling inequality.

\begin{proposition}\label{decosim2}
Let $P:\C^n\rightarrow X$ be an $m$-homogeneous tetrahedral polynomial. If $\xi=(\xi_1,\ldots,\xi_n)$ is a vector of independent symmetric random variables and $\xi^{(1)},\ldots,\xi^{(m)}$ are iid copies of $\xi$, then
\begin{multline*}
\frac{1}{|\Pi_{k,m}|} \sum_{\pi\in \Pi_{k,m}}
	(\E\|L_\pi(\xi^{(1)},\ldots,\xi^{(m)})\|^p)^{1/p} \\ \leq
	(\E\|P(\xi)\|^p)^{1/p} \\
	\leq  \frac{e^{m}}{|\Pi_{k,m}|} \sum_{\pi\in \Pi_{k,m}}
	(\E\|L_\pi(\xi^{(1)},\ldots,\xi^{(m)})\|^p)^{1/p}
\end{multline*}
for every $1 \leq p < \infty$.
\end{proposition}
\begin{proof}
First of all, let us note that, in a way, $L_\pi(\xi,\ldots,\xi)$ is already decoupled due to its algebraic structure since the index sets $B_1,\ldots, B_l$ never overlap. In fact, notice that replacing each coordinate $\xi_i$ with $\xi_i^{(l)}$ whenever $i\in B_l$ we deduce
\begin{equation} \label{almeida}
\E\Vert L_\pi(\xi,\ldots,\xi ) \Vert^{p} = \E\Vert L_\pi(\xi^{(1)},\ldots,\xi^{(m)}) \Vert^{p} \,.
\end{equation}

On the one hand, from~\eqref{eqdecom} and~\eqref{dalla}  we have
\begin{multline*}
(\E\|P(\xi)\|^p)^{1/p}
\leq \frac{e^{m}}{|\Pi_{k,m}|} \Big( \E \Big\| \sum_{\pi\in \Pi_{k,m}}
L_\pi(\xi,\ldots,\xi) \Big\|^p \Big)^{1/p}\\
	\leq  \frac{e^{m}}{|\Pi_{k,m}|} \sum_{\pi\in \Pi_{k,m}}
	(\E\|L_\pi(\xi,\ldots,\xi)\|^p)^{1/p} \,.
\end{multline*}
On the other hand, for each fixed partition $\pi$ we can use~\eqref{lpi} to have
\begin{equation} \label{lligassa}
\begin{split}
(\E_{\xi}\|L_\pi(\xi,\ldots,& \xi)\|^p)^{1/p}
=  \big(\E_{\xi}\big\|  \mathbb{E}_{\varepsilon} [\varepsilon_1\cdots\varepsilon_m P(T_\pi(\varepsilon) \xi) ] \big\|^p \big)^{1/p} \\
& \leq  \mathbb{E}_{\varepsilon}  \big(\E_{\xi}\|\varepsilon_1\cdots\varepsilon_m P(T_\pi(\varepsilon) \xi) \|^p \big)^{1/p}
=  \mathbb{E}_{\varepsilon}  \big(\E_{\xi}\| P(T_\pi(\varepsilon) \xi) \|^p \big)^{1/p}
\end{split}
\end{equation}
Now, since $\xi$ is symmetric, changing its sign does not affect its distribution, so that $T_\pi(\varepsilon) \xi \sim \xi$ and, then
\begin{multline*}
\frac{1}{|\Pi_{k,m}|} \sum_{\pi\in \Pi_{k,m}} (\E\|L_\pi(\xi,\ldots,\xi)\|^p)^{1/p}  \\
\leq \frac{1}{|\Pi_{k,m}|}\sum_{\pi\in \Pi_{k,m}}  \mathbb{E}_{\varepsilon}  \big(\E_{\xi}\| P( \xi) \|^p \big)^{1/p}
= (\E_{\xi} \|P(\xi)\|^p)^{1/p} \,.
\end{multline*}
This, in view of~\eqref{almeida}, completes the proof.
\end{proof}

With the same idea we can get a decoupling inequality involving convex functions, very much in the spirit of Kwapie\'n's result presented in~\eqref{decosim}.

\begin{proposition}\label{decosim3}
Let $P:\C^n\rightarrow X$ be an $m$-homogeneous tetrahedral polynomial and $\Phi:X\rightarrow \R_{\geq0}$ a convex function such that $\Phi(x)=\Phi(- x)$ for every $x\in X$.  If $\xi=(\xi_1,\ldots,\xi_n)$ is a vector of independent symmetric random variables and $\xi^{(1)},\ldots,\xi^{(m)}$ are iid copies of $\xi$ we have
	\[
	\E\Phi\Big(\frac{k^{m}}{\binom{km}{m}}P(\xi)\Big)
	\leq \frac{1}{|\Pi_{k,m}|} \sum_{\pi\in \Pi_{k,m}} \E\Phi(L_\pi(\xi^{(1)},\ldots,\xi^{(m)}))
	\leq \E\Phi(P(\xi)).
	\]
\end{proposition}
\begin{proof}
First of all the same argument as in~\eqref{almeida} shows that it is enough to check the inequalities for $\E\Phi(L_\pi(\xi,\ldots,\xi))$. Now, using the same argument as in~\eqref{lligassa}, this time with Jensen's inequality and the fact that $\Phi(x)=\Phi(- x)$, we get
\[
\E_\xi\Phi(L_\pi(\xi,\ldots,\xi))\leq \E_{\varepsilon,\xi}\Phi(\varepsilon_1\ldots\varepsilon_m P(T_\pi(\varepsilon) \xi))=\E_\xi\Phi(P(\xi)),
\]
for each  $\pi\in \Pi_{k,m}$. This gives one inequality. For the other one, note that, using again Jensen's inequality in~\eqref{eqdecom} yields
\[
\Phi\Big(\frac{k^{m}}{\binom{km}{m}}P(\xi)\Big)
\leq  \frac{1}{|\Pi_{k,m}|} \sum_{\pi\in \Pi_{k,m}} \Phi(L_\pi(\xi^{(1)},\ldots,\xi^{(m)})) \,.
\]
The linearity of $\mathbb{E}$ completes the proof.
\end{proof}

We finish this section by noting that Proposition~\ref{decosim2} cannot be deduced from Proposition~\ref{decosim3}, since the exponents $1/p$ would remain outside the average over $\pi\in \Pi_{k,m}$. In other words, the inequalities for $p$-norms are better than those for general convex functions.

\subsection{Comparison of random polynomials}

Having our new decomposition at hand, we are almost in position to prove Theorem~\ref{compara}. Let us state two results we use. The following lemma was proven in \cite[Lemma~4.2]{CaMaSe} and shows that Walsh and Steinhaus tetrahedral polynomials have equivalent $p$-norms (i.e. we may replace Steinhaus by Rademacher random variables) without assuming homogeneity or geometrical conditions on the Banach space.  Besides the argument given there, this can also be proven using \cite[Proposition~6.3.1]{KwWo92} and checking the hypothesis by hand.

\begin{lemma} \label{lemma1}
 Let $X$ be a Banach space and $1\leq p <\infty$. For every tetrahedral polynomial $P : \mathbb{C}^{n} \to X$ of degree $m$  we have
 \begin{equation}\label{sw}
  (1+\sqrt{2})^{-m}\left(\mathbb{E}\|P(\varepsilon)\|^p\right)^{1/p}
  \le \left(\mathbb{E}\|P(w)\|^p\right)^{1/p}
  \le (1+\sqrt{2})^{m}	\left(\mathbb{E}\|P(\varepsilon)\|^p\right)^{1/p}.
\end{equation}
\end{lemma}

The second result relates the norms of a polynomial and its homogeneous projection.
Given a polynomial $P(z) = \sum_{\vert \alpha \vert \leq m} x_{\alpha} z^{\alpha}$, for each $1 \leq k \leq m$ we consider the corresponding \emph{$k$-homogeneous projection} given by $$P_{k}(z) = \sum_{\vert \alpha \vert =k} x_{\alpha} z^{\alpha}.$$
The following proposition~can be found in \cite[Lemma~2]{Kw87} (see also \cite[Lemma~3.2.4]{DPGi99}). We also refer to \cite{CaMaSe} where the exponential growth of the constant on the degree of the polynomial (i.e., to be of the form $C^m$ for some $C\ge 1$) is explicitly derived.

\begin{proposition} \label{lemma3}
Let $X$ be a Banach space. There exists $C \geq 1$ so that for every $1\leq p<\infty$, every non-trivial symmetric random variable $\xi_{0}$ and every  tetrahedral polynomial $P : \mathbb{C}^{n} \to X$ of degree $m$ we have
	\[
	(\mathbb{E}\|P_k(\xi)\|^p)^{1/p}
	\leq C^m (\mathbb{E}\|P(\xi)\|^p)^{1/p} \,,
	\]
where $\xi $  is a random vector of iid copies of $\xi_0$.
\end{proposition}

We finally have come to the point where we can  prove  Theorem~\ref{compara}. We see that Steinhaus (or Walsh) polynomials always have the smallest $p$-norms (up to a constant $C^m$) compared to polynomials on other symmetric random variables and that, under certain assumptions, they are even equivalent (again up to $C^{m}$ constant).

\begin{proof}[Proof of Theorem~\ref{compara}]
The proof of~\ref{landa1} is essentially an adaptation of \cite[Proposition~12.11]{DiJaTo95}. Note that $\xi = (\xi_{1}, \ldots , \xi_{n})$ consists of independent copies of $\xi_{0}$. Then, given $A \subseteq [n]$ with $\vert A \vert =m$ and $z\in \C^n$, we have
\[
(\E_\xi|\xi_{0}|)^m z_A=\prod_{i\in A}\E_\xi|\xi_{i}|z_{i}=\E_\xi \prod_{i\in A}|\xi_{i}|z_{i}=\E_\xi(|\xi|z)_A \,.
\]
As a consequence, if  $P(z)=\sum_{|A|=m}x_A z_A$ is $m$-homogeneous, we have
\begin{align*}
	(\E_\xi|\xi_{0}|)^{pm}\mathbb{E}_w\|P(w)\|^p
	&= \mathbb{E}_w\Big\|\sum_{|A|=m} x_A (\E_\xi|\xi_{0}|)^mw_A\Big\|^p	
	\\&= \mathbb{E}_w\Big\|\sum_{|A|=m} x_A \E_\xi(|\xi|w)_A \Big\|^p
	= \mathbb{E}_w\| \E_\xi P(|\xi|w)\|^p
	\\ &\le \E_\xi \E_w\|  P(|\xi|w)\|^p = \E_\xi\E_w\|  P(\xi w)\|^p,
\end{align*}
for $1 \leq p < \infty$ (where in the last step we used the rotation invariance of $w$).
We can now use Lemma~\ref{lemma1}
to replace Steinhaus with Rademacher variables, obtaining
\[
(\E_\xi|\xi_1|)^{pm}\mathbb{E}_w\|P(w)\|^p
\leq (1+\sqrt{2})^{pm} \E_\xi\E_\varepsilon\|  P(\xi \varepsilon)\|^p \,.
\]
Since $\xi_{0}$ is symmetric we have that $\xi\varepsilon\sim \xi$,
which yields
\begin{equation} \label{lvb}
\Big(\frac{\E |\xi_{0}|}{1+\sqrt{2}} \Big)^{m} \big(\mathbb{E}\|P(w)\|^p \big)^{1/p}
\leq  \big(\E\|P(\xi)\|^p \big)^{1/p} \,,
\end{equation}
and~\eqref{walshchica} holds for homogeneous tetrahedral polynomials (note that, being $\xi_{0}$ non-trivial, $\mathbb{E}\vert \xi_{0} \vert  \neq 0$).\\
Given an arbitrary tetrahedral polynomial $P$, we split it on its homogeneous components and apply Proposition~\ref{lemma3} to get
\begin{equation} \label{bernal}
\begin{split}
(\E\|P(w)\|^p)^{1/p} & \le \sum_{k=0}^m (\E\|P_k(w)\|^p)^{1/p}\le \sum_{k=0}^m C^k (\E\|P_k(\xi)\|^p)^{1/p}
\\ & \le \widetilde{C}^m \sum_{k=0}^m C^k (\E\|P(\xi)\|^p)^{1/p}
\le (2C\widetilde{C})^m (\E\|P(\xi)\|^p)^{1/p} .
\end{split}
\end{equation}

To show~\ref{landa2} we again start with the homogeneous case. Let us observe first that by Lemma~\ref{lemma1}
it is enough to see that
\begin{equation} \label{quintana}
\big(\mathbb{E}\|P(\xi)\|^p \big)^{1/p}\leq C^m \big(\E\|P(\varepsilon)\|^p \big)^{1/p} \,.
\end{equation}
A simple computation from Proposition~\ref{decosim2} shows that if
\begin{equation} \label{parra}
\big(\mathbb{E}\|L(\xi^{(1)},\ldots,\xi^{(m)})\|^p\big)^{1/p}
\leq C^m \big( \E \|L(\varepsilon^{(1)},\ldots,\varepsilon^{(m)} ) \|^p \big)^{1/p}
\end{equation}
holds for every $m$-linear $L : \mathbb{C}^{n} \times \cdots \times \mathbb{C}^{n} \to X$, then~\eqref{quintana} holds true. It suffices, therefore, to show that~\eqref{parra} holds, and we do this by induction. For the case $m=1$, from
\cite[Proposition~3.2]{pisier} (see also \cite[Proposition~9.14]{contrac}
we know that, if $\xi_{0}$ is real (besides having finite $s$-norm),  we have
\[
\Big(\mathbb{E}\Big\|\sum_{j=1}^n x_j \xi_j \Big\|^p\Big)^{1/p}
\leq C \Big(\E\Big\|\sum_{j=1}^n x_j \varepsilon_j \Big\|^p \Big)^{1/p}
\]
for every choice of vectors $\{x_j\}_{j=1}^n\subseteq X$. This immediately generalizes to complex random vectors by splitting $\xi$ in its real and imaginary parts and using the triangle inequality, so the case $m=1$ holds. The rest follows easily by induction on $m$, showing that~\eqref{propcomparacion} holds for homogeneous tetrahedral polynomials. The inequality for arbitrary tetrahedral polynomials follows as in~\eqref{bernal}.
\end{proof}

We end this section with one further result for Rademacher and Steinhaus random variables.
Remark~\ref{remate} shows that there is no hope to obtain a decoupling inequality as~\eqref{decoupsim} for polynomials taking values in some space with trivial cotype, not even for random variables as nice as Rademacher (or Steinhaus). More precisely,
there is no hope to obtain an inequality as~\eqref{bluemoon} on such spaces. However, using the idea in~\eqref{ecdecosuma}, obtaining a decoupling inequality between $P$ and $M$ is essentially the same as comparing
the moments of the polynomial evaluated in $\xi$ and $\sum_{l=1}^m \xi^{(l)}$. For Rademacher and Steinhaus variables this allows us to show that the left-hand side of~\eqref{decoupsim} (i.e. the reverse inequality to that in
\eqref{bluemoon}) holds even for spaces with trivial cotype.

\begin{proposition} \label{bennett}
  	Let $X$ be a Banach space. There is a constant $C\ge1$ such that
  	\[
  	(\E\|P(\varepsilon)\|^p)^{1/p}
  	  	\le  C^m m^{m/2} (\E\|M(\varepsilon^{(1)},\ldots,\varepsilon^{(m)})\|^p)^{1/p}
  	\]
  	and
  	\[
  	(\E\|P(w)\|^p)^{1/p}
  	  	\le  C^m m^{m/2} (\E\|M(w^{(1)},\ldots,w^{(m)})\|^p)^{1/p}
  	\]
  	for every $1\leq p <\infty$ and every $m$-homogeneous tetrahedral polynomial $P:\C^n\rightarrow X$.
\end{proposition}
\begin{proof}
Let us fix some $1 \leq p < \infty$ and note that, in view of Lemma~\ref{lemma1},
it is enough to check that the inequality holds for Steinhaus random variables.
Let $P$ be an $m$-homogeneous tetrahedral polynomial. Applying~\eqref{lvb} to the random variable $\sum_{l=1}^m w^{(l)}$ and~\eqref{ecdecosuma} we get
\begin{align*}
\left(\frac{\E\left|\sum_{l=1}^m w_1^{(l)}\right|}{1+\sqrt{2}}\right)^m(\E\|P(w)\|^p)^{1/p}&\le  \Big(\E\Big\|P\Big(\sum_{l=1}^m w^{(l)}\Big)\Big\|^p\Big)^{1/p}
\\ &\le m^m (\E\|M(w^{(1)},\ldots,w^{(m)})\|^p)^{1/p}.
\end{align*}
Now by Khinchin inequality (or Theorem~\ref{poiss}) we have
\[\sqrt{m}=\Big(\E\Big|\sum_{l=1}^m w_1^{(l)}\Big|^2\Big)^{1/2}\le \sqrt{2} \E\Big|\sum_{l=1}^m w_1^{(l)}\Big|. \]
Joining both inequalities we conclude
\[(\E\|P(w)\|^p)^{1/p}
  	\le  (\sqrt{2}+2)^m m^{m/2} (\E\|M(w^{(1)},\ldots,w^{(m)})\|^p)^{1/p}. \qedhere \]
\end{proof}

\section{Geometric conditions for full independence}\label{sectotin}

So far we have been able to compare the norm of polynomials with different random variables for tetrahedral polynomials (see Theorem~\ref{compara} or Lemma~\ref{lemma1}) and then to compare a tetrahedral random polynomial with its multilinear counterpart. A question one might ask is if we can do better: can we compare the random polynomial with a fully independent sum (and not just the \textit{less dependent sum} given by the multilinear operator)? A second question raises naturally: what can be said for non-tetrahedral polynomials?. We show that under stronger geometric conditions that include Banach lattices of non-trivial type we can compare random polynomials with independent randoms sums. In addition, for Steinhaus variables our estimates hold for arbitrary (not necessarily tetrahedral) polynomials.

We begin with two results that are essentially a consequence of \cite[Theorem~12.2]{TJ} and were shown in \cite[Theorem~4.1]{ruc} for specific random variables. The proofs for general random variables are analogous. We include them for the sake of completeness and to lay the ground for Theorems \ref{propgap} and \ref{propgapduo}.

\begin{proposition} \label{RUCvectorvaluedA}
	Let $X$ be a Banach space of type $2$. There is a constant $C>0$ such that for every orthonormal sequence of (not necessarily independent) random variables $(\xi_i)_{i\in \N}\subseteq L^2(\mu)$
	and every choice of finitely many $x_1, \ldots, x_n \in X$ we have
	\begin{equation*}
	\Big(\mathbb{E} \Big\|\sum_{i=1}^n \varepsilon_i  x_i \Big\|^2\Big)^{1/2}
	\leq C \Big(\E\Big\|\sum_{i=1}^n \xi_i x_i \Big\|^2\Big)^{1/2}.
	\end{equation*}
\end{proposition}

To prove this we need the concept of $p$-summing operator. For $1\le p\le \infty$ an operator $T:X\rightarrow Y$ is said to be \emph{$p$-summing} if there is a constant $C\geq 1$ such that for every choice of vectors $x_1,\ldots x_n\in X$ we have
\[
\Big(\sum_{i=1}^n \|T(x_i)\|^p_Y\Big)^{1/p}
\le C \sup_{x^*\in B_{X^*}}\Big(\sum_{i=1}^n |x^*(x_i)|^p\Big)^{1/p} \,.
\]
We denote by $\pi_p(T)$ the smallest possible $C$.
We refer the reader to \cite[Chapter~2]{DiJaTo95} and \cite[Chapter~2]{TJ} for a detailed exposition.

\begin{proof}[Proof of Proposition~\ref{RUCvectorvaluedA}]
Let $T: \ell_2^n \to X$ be the operator defined by $T(e_i)=x_i$. Notice that combining Lemma~\ref{lemma1}
 and~\eqref{walshchica} we get
\begin{equation*}
\Big(\mathbb{E} \Big\|\sum_{i=1}^n \varepsilon_i  x_i \Big\|^2\Big)^{1/2}
	\le C
	\Big(\mathbb{E} \Big\|\sum_{i=1}^n \gamma_i   x_i \Big\|^2 \Big)^{1/2}.
\end{equation*}	
On the other hand, we know from \cite[Theorem~12.2]{TJ} that if $X$ has type $2$, then
\begin{equation}\label{ruc2}	
\Big(\mathbb{E} \Big\|\sum_{i=1}^n \gamma_i  x_i \Big\|^2 \Big)^{1/2} \le \widetilde{C} \pi_2(T^\ast).
\end{equation}	
So, to complete the proof it is only left to find a convenient bound for $\pi_2(T^\ast)$. To find it let us observe that the orthogonality of $(\xi_{i})_{i}$ yields
\[
\Big\|\sum_{i=1}^n \xi_i x^*(x_i)\Big\|_{L^2(\mu)}=\Big(\sum_{i=1}^n |x^*(x_i)|^2\Big)^{1/2}
	= \|(x^*(x_i))_{i=1}^n\|_{\ell_2^n}
	=  \|T^*(x^*)\|_{\ell_2^n}
\]
for every $x^*\in X^*$. Therefore, given  a finite collection of vectors $x_k^*\in X^*$ we have
\begin{align*}
	\sum_k \|T^*(x_k^*)\|_{\ell_2^m}^2
	&=\sum_k \Big\|\sum_{i=1}^n \xi_i x_k^*(x_i)\Big\|_{L^2(\mu)}^2
	=\E\Big[\sum_k \Big|x_k^*\Big(\sum_{i=1}^n \xi_i x_i\Big)\Big|^2\Big]
	\\ & \le  \E\Big\|\sum_{i=1}^n \xi_i x_i \Big\|^2 \sup_{x^{**}\in B_{X^{**}}}\sum_k |x^{**}(x_k^*)|^2.
\end{align*}
This gives
\[
\pi_2(T^\ast)\leq  \Big(\E\Big\|\sum_{i=1}^n \xi_i x_i \Big\|^2\Big)^{1/2}  \,,
\]
which completes the proof.
\end{proof}

With a similar argument we have the following dual statement for spaces with cotype $2$.

\begin{proposition} \label{RUCvectorvaluedA2}
	Let $X$ be a Banach space of cotype $2$. There is a constant $C>0$ such that for every orthonormal sequence of (not necessarily independent) random variables $(\xi_i)_{i\in \N}\subseteq L^2(\mu)$
	and every choice of finitely many $x_1, \ldots, x_n \in X$ we have
	\begin{equation*}
	\Big(\E\Big\|\sum_{i=1}^n \xi_i x_i \Big\|^2\Big)^{1/2}
	\leq C \mathbb{E} \Big(\Big\|\sum_{i=1}^n \varepsilon_i  x_i \Big\|^2\Big)^{1/2}.
	\end{equation*}
\end{proposition}
\begin{proof}
As before let $T: \ell_2^n \to X$ be the operator defined by $T(e_i)=x_i$. 	Since $X$ has cotype $2$, using again \cite[Theorem~12.2]{TJ} as well as \cite[Theorem~12.27]{DiJaTo95} (which is the linear case of Theorem~\ref{compara} above), we deduce
\[
	\pi_2(T) \le C \Big(\mathbb{E} \Big\|\sum_{i=1}^n \gamma_i  x_i \Big\|^2 \Big)^{1/2}
	\le \widetilde{C}\Big(\mathbb{E} \Big\|\sum_{i=1}^n \varepsilon_i  x_i \Big\|^2\Big)^{1/2}.
\]
To finish the proof we just have to find a convenient lower bound for $\pi_{2}(T)$. To this aim, we consider the operator $S: X^* \to L^2(\mu)$ given by
\[
S(x^*)=\sum_{i=1}^n \xi_i x^*(x_i).
\]
We have the following commutative diagram
\[
\begin{tikzcd}
L^2(\mu) \arrow[r,"S^*"] \arrow[d,swap,"q"] &  X^{**}  \\
\ell_2^n \arrow[r,"T"] & X \arrow[hookrightarrow,swap]{u}{i}
\end{tikzcd}
\]
where $q: L^2(\mu) \rightarrow \ell_2^n$ is the projection given by $q(f)=(\langle f,\xi_i \rangle)_{i=1}^n$ for every $f\in L^2(\mu)$ and $i: X\hookrightarrow X^{**}$ is the natural inclusion. Since clearly $\|i\|=\|q\|=1$, the ideal property of $2$-summing operators \cite[2.4]{DiJaTo95} gives $\pi_2(S^*) \leq \pi_2(T)$. Finally,  \cite[Theorem~2]{Kw70} (see also \cite[Proposition~1.1]{Pi78} and \cite[Corollary~5.21]{DiJaTo95}) give
\[
\Big(\E\Big\|\sum_{i=1}^n \xi_i x_i \Big\|^2\Big)^{1/2}\le \pi_2(S^*).
\]
This completes the proof.
\end{proof}

The Steinhaus monomials $(w^{\alpha})_{\alpha \in \mathbb{N}_{0}^{n}} \subseteq L^{2} (\mathbb{T}^{n})$ form an orthonormal system, and we may apply the previous results to compare polynomials to sums of independent Rademacher variables. Let us note that, for each $\alpha$, on the side of the
Rademacher random variables we get a different (independent) random variable. Since we have already used $\varepsilon^{\alpha}$ and $\varepsilon_A$ for the product of
components of a random vector of copies of $\varepsilon$ (and this is not what we get here) we use $\varepsilon_{i_{\alpha}}$ (with some injection $\alpha\mapsto i_\alpha\in\N$, stressing the fact that for each $\alpha$ we have an individual random variable). In other words, for the (finite) set
$\{ \alpha \in \mathbb{N}_{0}^{n} \colon \vert \alpha \vert \leq m  \}$ we consider the family of independent identically distributed Rademacher random variables $\{\varepsilon_{i_{\alpha}} \}_{\alpha}$. With this notation at hand, we recover the following from \cite[Theorem~4.1]{ruc}.

\begin{corollary} \label{marceneiro}
Let $X$ be a Banach space.
\begin{enumerate}[(a)]
\item If $X$ has type $2$, then there is a constant $C \geq 1$ such that, for every finite choice of vectors $\{x_\alpha\}_{|\alpha|\le m}$ we have
\[
\Big(\mathbb{E} \Big\| \sum_{|\alpha|\le m}   \varepsilon_{i_\alpha} x_\alpha  \Big\|^2\Big)^{1/2}
	\leq C
\Big( \mathbb{E} \Big\| \sum_{|\alpha|\le m} x_\alpha w^\alpha  \Big\| ^{2} \Big)^{1/2} \,.
\]

\item If $X$ has cotype $2$, then there is a constant $C \geq 1$ such that, for every finite choice of vectors $\{x_\alpha\}_{|\alpha|\le m}$ we have
\[
\Big( \mathbb{E} \Big\| \sum_{|\alpha|\le m} x_\alpha w^\alpha  \Big\| ^{2} \Big)^{1/2}
\leq C
\Big(\mathbb{E} \Big\| \sum_{|\alpha|\le m} 	
  \varepsilon_{i_\alpha} x_\alpha  \Big\|^2\Big)^{1/2}
\,.
\]
\end{enumerate}
\end{corollary}

Let us note that an analogous result holds for any sequence of characters in the context of Fourier analysis on groups.

We also mention that if $\xi_{0} \in L^{2} (\mu)$ and $\xi$ is a random vector consisting of $n$ independent copies of $\xi_{0}$, then the family $\{ \xi_{A} \colon A \subseteq [n] , \,  \vert A \vert \leq m  \}$ is
orthogonal. Then orthonormality is achieved just by normalising, and we can apply Propositions~\ref{RUCvectorvaluedA} and~\ref{RUCvectorvaluedA2}. Again, for each $A$ we write $\varepsilon_{i_A}$
for an independent copy of a Rademacher random variable (not to be confused with the Walsh monomial $\varepsilon_{A}$). Noting that
\[
\|\xi_A\|_2=\Big(\E\Big|\prod_{i\in A}\xi_i\Big|^2\Big)^{1/2}=\|\xi_0\|_2^{|A|},
\]
we deduce the following.

\begin{corollary}\label{cortipo2}
Let $X$ be a Banach space.
\begin{enumerate}[(a)]
\item If $X$ has  type 2, then there is a constant $C \geq 1$ such that for every non-trivial symmetric random variable $\xi_0\in L^2(\mu)$ and every choice of vectors $\{x_A\}_{|A|\le m}$ we have
\begin{equation*}
\Big(\mathbb{E} \Big\|\sum_{|A|\le m}     \varepsilon_{i_A} \|\xi_0\|_2^{|A|} x_A  \Big\|^2\Big)^{1/2}
	\le C
	\Big(\mathbb{E} \Big\| \sum_{|A|\le m} x_A \xi_A   \Big\|^2 \Big)^{1/2} \,.
\end{equation*}

\item If $X$ has cotype 2, then there is a constant $C \geq 1$ such that for every symmetric random variable $\xi_0\in L^2(\mu)$ and every choice of vectors $\{x_A\}_{|A|\le m}$ we have
\begin{equation*}
\Big(\mathbb{E} \Big\| \sum_{|A|\le m} x_A \xi_A   \Big\|^2 \Big)^{1/2}
	\le C \Big(\mathbb{E} \Big\|\sum_{|A|\le m}  \varepsilon_{i_A} \|\xi_0\|_2^{|A|}  x_A  \Big\|^2\Big)^{1/2}.
\end{equation*}	
\end{enumerate}
\end{corollary}

Let us observe that using the contraction principle in the previous inequalities immediately gives
\begin{align*}
\Big(\mathbb{E} \Big\| \sum_{|A|\le m}   \varepsilon_{i_A} x_A  \Big\|^2\Big)^{1/2}
	&\le C \max\{1,\|\xi_0\|_2^{-m}\}
	\Big(\mathbb{E} \Big\| \sum_{|A|\le m} x_A \xi_A   \Big\|^2 \Big)^{1/2},
	\intertext{and}
 \Big(\mathbb{E} \Big\| \sum_{|A|\le m} x_A \xi_A   \Big\|^2 \Big)^{1/2}
	&\le C \max\{1, \|\xi_0\|_2^m\} \Big(\mathbb{E} \Big\|   \sum_{|A|\le m}  \varepsilon_{i_A} x_A  \Big\|^2\Big)^{1/2}.
\end{align*}

Having type or cotype 2 are quite restrictive geometric conditions for a Banach space. What is more, if one is looking for a decoupling inequality that estimates the norm of a polynomial from above and below one needs that the Banach space $X$ enjoys both type and cotype 2. In other words, by a fundamental result of Kwapie\'{n} proved in \cite{Kwapo}, $X$ must by isomorphic to a Hilbert space where these estimates hold trivially.
We try now to find a weaker condition on the space that (allowing exponential dependence on $m$) still provides an analogue result to Corollary~\ref{marceneiro}. We find it in
the \emph{Gaussian Average Property (GAP)} introduced in \cite{CaNi97}. A Banach space $X$ has GAP if there exists $C\geq1$ such that for every finite choice $x_{1}, \ldots , x_{n} \in X$, the operator $T : \ell_{2}^n  \to X^{*}$ defined by $T(e_i) = x_{i}$ satisfies
\begin{equation} \label{gap}
\Big( \E \Big\Vert \sum_{i=1}^{n} x_{i} \gamma_{i} \Big\Vert^{2} \Big)^{1/2} \leq C \pi_{1}(T^*) .
\end{equation}
By \cite[Theorem~1.4]{CaNi97} Banach spaces with type 2 and Banach lattices with finite cotype have GAP. In turn, spaces with GAP have finite cotype (see \cite[Theorem~1.3]{CaNi97}). Finally, we mention that GAP is closely related to the Gordon-Lewis property and the concept of local unconditional structure.\\

The key step in the proof of Proposition~\ref{RUCvectorvaluedA} was~\eqref{ruc2}, which is very similar to~\eqref{gap}. This is, then, our main tool now. However, since~\eqref{gap} involves the 1-summing norm instead of the 2-summing norm in~\eqref{ruc2} we need a Khinchin type inequality to hold. Fortunately we have it in Theorem~\ref{poiss}.
With this we can now get the inequalities we were aiming at. We follow the lines of \cite[Theorem~1.1]{Pi78} (see also \cite[Lemma~2.2]{CaDeSe16}).

\begin{theorem}\label{propgap}
If a Banach space $X$ has GAP, then there is a constant $C\geq 1$ such that for every choice of vectors $\{x_\alpha\}_{|\alpha|\le m}$ we have
\begin{align}\label{ecgap2}
\Big(\mathbb{E} \Big\| \sum_{|\alpha|\le m}   \varepsilon_{i_\alpha} x_\alpha  \Big\|^2\Big)^{1/2}
	\le C \,
	\mathbb{E} \Big\| \sum_{|\alpha|\le m}\sqrt{2}^{|\alpha|} x_\alpha w^\alpha  \Big\| \,.
\end{align}
\end{theorem}
\begin{proof}
The proof is very similar to that of Proposition~\ref{RUCvectorvaluedA}. Let $\Lambda = \{ \alpha \in \mathbb{N}_{0}^{n} \colon \vert \alpha \vert \leq m  \}$ and, as before, define the operator $T: \ell_2(\Lambda) \to X$ by $T(e_\alpha)=x_\alpha$.
Combining Lemma~\ref{lemma1}, ~\eqref{walshchica} and~\eqref{gap} we have
\begin{equation*}
\Big(\mathbb{E} \Big\| \sum_{|\alpha|\le m}  \varepsilon_{i_\alpha} x_\alpha \Big\|^2\Big)^{1/2}
	\le C
	\Big(\mathbb{E} \Big\| \sum_{|\alpha|\le m}  \gamma_{i_\alpha} x_\alpha \Big\|^2 \Big)^{1/2}
\le \widetilde{C} \pi_1(T^\ast).
\end{equation*}	
We look now for a good upper bound for $\pi_1(T^\ast)$. First of all observe that, for each  $x^*\in X^*$ we have
\[
\Big(\mathbb{E}\Big|\sum_{|\alpha|\le m}  x^*(x_\alpha) w^\alpha \Big|^2 \Big)^{1/2}
	= \|(x^*(x_\alpha))_{|\alpha|\le m}\|_{\ell_2(\Lambda)}
=  \|T^*(x^*)\|_{\ell_2(\Lambda)}.
\]
With this, given  a finite collection of vectors $x_i^*\in X^*$ and using  Theorem~\ref{poiss} we get
\begin{align*}
\sum_{i=1}^N \|& T^*(x_i^*)  \|_{\ell_2(\Lambda)}
	=\sum_{i=1}^N \Big(\mathbb{E}\Big|\sum_{|\alpha|\le m}  x_i^*(x_\alpha) w^\alpha \Big|^2 \Big)^{1/2}
\le  \sum_{i=1}^N \mathbb{E}\Big|\sum_{|\alpha|\le m} \sqrt{2}^{|\alpha|} x_i^*(x_\alpha) w^\alpha \Big| \\
&	=  \mathbb{E} \sum_{i=1}^N\Big|\sum_{|\alpha|\le m}  \sqrt{2}^{|\alpha|} x_i^*(x_\alpha) w^\alpha \Big|
\le  \E\Big\|\sum_{|\alpha|\le m} \sqrt{2}^{|\alpha|} x_\alpha w^\alpha \Big\| \sup_{x^{**}\in B_{X^{**}}}\sum_{i=1}^N |x^{**}(x_i^*)|.
\end{align*}
The definition of the 1-summing norm gives
\[
\pi_1(T^\ast)\leq  \E\Big\|\sum_{|\alpha|\le m}  \sqrt{2}^{|\alpha|} x_\alpha w^\alpha  \Big\|\,,
\]
and completes the proof.
\end{proof}

We can use a similar argument to get a dual result. We recall first two concepts that are going to play a role on the proof. On the one hand, an operator $T:X\rightarrow Y$ is said to be \emph{factorable}  if there is a measure space $(\Omega,\Sigma,\mu)$ and operartors $A:L^\infty(\mu)\rightarrow Y^{**}$ and $B:X\rightarrow L^\infty(\mu)$ such that $i T= AB$ where $i:Y\hookrightarrow Y^{**}$ is the natural inclusion. In other words, $iT$ factors through $L^\infty(\mu)$, and we have the following commutative diagram
\begin{equation} \label{marais}
\begin{tikzcd}
 X \arrow[r,"T"] \arrow[rd,swap,"B"] & Y \arrow[r,"i"] &  Y^{**}  \\
& L^\infty(\mu) \arrow[ru,swap,"A"]  &
\end{tikzcd}
\end{equation}
We write
\[
\gamma_\infty(T)=\inf \|A\|\|B\|,
\]
where the infimum is taken over all the possible factorizations.
We refer the reader to \cite[Chapters~7 and~9]{DiJaTo95} for a detailed exposition. \\
On the other hand, a Banach space $X$ is said to be \emph{$K$-convex} if the 1-homogeneous projection is bounded on $L^{2} (\{-1,1\}^{\infty}, X)$. More precisely,
the mapping defined on Walsh polynomials by
\[R \Big( \sum_{A } x_{A} \varepsilon_{A} \Big) = \sum_{\vert A \vert =1} x_{A} \varepsilon_{A},\]
extends to a bounded operator
$R:L^{2} (\{-1,1\}^{\infty}, X)\to L^{2} (\{-1,1\}^{\infty}, X)$ (which is known as the Rademacher projection). A simple duality argument shows that $X$ is $K$-convex if and only if $X^*$ is $K$-convex. Also, a very important result states that a Banach space is $K$-convex if and only if it has non-trivial type (see e.g. \cite[Theorem~13.3]{DiJaTo95}). \\
Finally, we know from \cite[Theorem~1.7]{CaNi97} that a Banach space $X$ is of finite cotype and its dual has GAP if and only if $X$ is $K$-convex and there is a constant $C\geq 1$ such that for every $T:\ell_2^n\rightarrow X$ defined by $T(e_i)=x_i$ we have
\begin{equation}\label{gapdual}
C^{-1} \gamma_\infty (T) \le \Big( \E \Big\Vert \sum_{i=1}^{n} x_{i} \gamma_{i} \Big\Vert^{2} \Big)^{1/2} \le C \gamma_\infty (T).
\end{equation}

\begin{theorem}\label{propgapduo}
Let $X$ be a Banach space of finite cotype such that $X^*$ has GAP and fix $q>\cot(X)$. There is a constant $C\geq 1$ such that for every finite choice of vectors $\{x_\alpha\}_{|\alpha|\le m}$ we have
\begin{equation}\label{ecgapdu}
\Big(\mathbb{E} \Big\| \sum_{|\alpha|\le m} x_\alpha w^\alpha  \Big\|^q\Big)^{1/q}
	\le C\Big(\mathbb{E} \Big\| \sum_{|\alpha|\le m}   \varepsilon_{i_\alpha} \sqrt{\frac{q}{2}}^{|\alpha|} x_\alpha  \Big\|^2\Big)^{1/2}.
\end{equation}
\end{theorem}
\begin{proof}
Similarly as in Proposition~\ref{RUCvectorvaluedA2}, we define the operator $S: X^* \to L^q(\T^n)$ by
\[
S(x^*)=\sum_{|\alpha|\le m}  x^*(x_\alpha) w^\alpha.
\]
From \cite[Theorem~2]{Kw70} (see also \cite[Proposition~1.1]{Pi78} and \cite[Corollary~5.21]{DiJaTo95}) we have
\begin{equation*}
\Big(\mathbb{E} \Big\| \sum_{|\alpha|\le m} x_\alpha w^\alpha  \Big\|^q\Big)^{1/q}\le \pi_q(S^*) \,,
\end{equation*}
and the challenge is to bound properly $\pi_q(S^*)$. We do this in several stages, the first one being to bound it by $\gamma_{\infty}(S^{*})$. There are two things to be observed first. One, that we can restrict the codomain of $S^*$ to $X$ since the image of $S^*$ actually
lies on $X$ and its easy to check that the $\pi_q$-norm remains the same. The second one, that we know from \cite[Proposition~1.4]{MyP} (see also \cite[Theorem~11.14]{DiJaTo95}) that there is a constant $C\geq 1$ such that
\[
\pi_q(U)\le C \|U\|
\]
for every $U:L^\infty(\mu)\rightarrow X^{**}$  (here we are using that, by  the principle of local reflexivity, $X$ and $X^{**}$ share the same cotype; see \cite[Theorem~11.2.4]{kalton}). Then, if we factor $iS^*=AB$ through some $L^\infty(\mu)$ as in~\eqref{marais}, this and the  ideal property of $q$-summing operators give
\[
\pi_q(S^*)=\pi_q(iS^*)=\pi_q(AB)\leq \pi_q(A)\|B\|\le C \|A\|\|B\| \,,
\]
hence
\[
\pi_q(S^*)\le C \gamma_\infty(S^*).
\]
The second stage is to bound $\gamma_\infty(S^*)$. Let $R:L^{q^*}(\T^n)\rightarrow \ell_2(\Lambda)$ and $T: \ell_2(\Lambda) \to X$ be the operators defined by $R(f)=(\sqrt{2/q}^{|\alpha|}\widehat{f}(\alpha))_\alpha$ and $T(e_\alpha)=  \sqrt{q/2}^{|\alpha|}x_\alpha$ respectively. We have the following commutative diagram
\[
\begin{tikzcd}
L^{q^*}(\mu) \arrow[r,"S^*"] \arrow[d,swap,"R"] &  X^{**}  \\
\ell_2(\Lambda) \arrow[r,"T"] & X \arrow[hookrightarrow,swap]{u}{i}
\end{tikzcd}
\]
By the ideal property of factorable operators \cite[Theorem~7.1]{DiJaTo95} we get
\begin{equation*}
\gamma_\infty(S^*) = \gamma_\infty(iTR) \le
\gamma_\infty(T)\|R\|.
\end{equation*}
So, it only remains to bound these two factors. We begin by showing that $R$ is a contraction. Given $f\in L^{q^*}(\T^n)$ we can use Theorem~\ref{poiss} to get
\begin{align*}
\|R(f)\|_{\ell_2(\Lambda)}^2 &=\sum_{|\alpha|\le m} \Big(\frac{2}{q}\Big)^{|\alpha|} |\widehat{f}(\alpha)|^2=\E\Big[f(w)\sum_{|\alpha|\le m} \Big(\frac{2}{q}\Big)^{|\alpha|} \overline{\widehat{f}(\alpha)} \overline{w}^\alpha\Big]
\\ &\le \|f\|_{L_{q^*}} \Big\| \sum_{|\alpha|\le m} \Big(\frac{2}{q}\Big)^{|\alpha|} \widehat{f}(\alpha) w^\alpha\Big\|_{L^{q}}
\le  \|f\|_{L_{q^*}} \Big\| \sum_{|\alpha|\le m} \sqrt{\frac{2}{q}}^{|\alpha|} \widehat{f}(\alpha) w^\alpha\Big\|_{L^{2}}
\\ &=  \|f\|_{L_{q^*}} \|R(f)\|_{\ell_2(\Lambda)}.
\end{align*}
And, since $f$ was arbitrary, this gives $\Vert R \Vert \leq 1$.\\
Finally,  applying~\eqref{gapdual} and \cite[Theorem~12.27]{DiJaTo95} (which is the linear case of Theorem~\ref{compara} above) we obtain
\[
\gamma_\infty(T)\le \widetilde{C} \Big(\mathbb{E} \Big\| \sum_{|\alpha|\le m}   \gamma_{i_\alpha} \sqrt{\frac{q}{2}}^{|\alpha|} x_\alpha  \Big\|^2\Big)^{1/2}
\le \dbtilde{C} \Big(\mathbb{E} \Big\| \sum_{|\alpha|\le m}   \varepsilon_{i_\alpha} \sqrt{\frac{q}{2}}^{|\alpha|} x_\alpha  \Big\|^2\Big)^{1/2}\,,
\]
and this completes the proof.
\end{proof}

We finish this section with two inequalities in this context for tetrahedral polynomials.

\begin{corollary} \label{corgap}
\text{}
\begin{enumerate}[(a)]
\item Let $X$ be a Banach space with GAP. Then there is a constant $C\geq 1$ such that for every finite choice of vectors $\{x_A\}_{|A|\le m}$ and every  symmetric non-trivial random variable $\xi_0$  we have
\[
 \Big(\mathbb{E} \Big\|   \sum_{|A|\le m}  \varepsilon_{i_A} x_A  \Big\|^2\Big)^{1/2}
	\le C^m \mathbb{E} \Big\| \sum_{|A|\le m} x_A \xi_A   \Big\|.
\]

\item Let $X$ be a Banach space of finite cotype such that $X^*$ has GAP, let $\xi_0$ be a symmetric random variable with finite $s$-norm for some $s> \cot(X)$ and fix $1\le p <s$. There is a constant $C\geq 1$ such that for every finite choice of vectors $\{x_A\}_{|A|\le m}$ we have
\[
 \Big(\mathbb{E} \Big\| \sum_{|A|\le m} x_A \xi_A   \Big\|^p\Big)^{1/p}
	\le C^m \Big(\mathbb{E} \Big\|   \sum_{|A|\le m}  \varepsilon_{i_A} x_A  \Big\|^2\Big)^{1/2}\,.
\]
\end{enumerate}
\end{corollary}
\begin{proof}
For the first statement just observe that using the contraction principle (see Theorem~\ref{contr}) and~\eqref{ecgap2} we get
\begin{align*}
    \Big(\mathbb{E} \Big\| \sum_{|\alpha|\le m}   \varepsilon_{i_\alpha} x_\alpha  \Big\|^2\Big)^{1/2}
&\leq C\sqrt{2}^m \Big(\mathbb{E} \Big\| \sum_{|\alpha|\le m}   \varepsilon_{i_\alpha} \sqrt{2}^{-|\alpha|} x_\alpha  \Big\|^2\Big)^{1/2}
\\ &\leq C\sqrt{2}^m \E\Big\|\sum_{|\alpha|\le m} x_\alpha w^\alpha  \Big\|.
\end{align*}
Combining this with Theorem~\ref{compara}~\ref{landa1} proves the claim.
On the other hand, choosing a suitable $q$ (for example $q=s$) and applying the contraction principle  to~\eqref{ecgapdu} we deduce
\[
\Big(\mathbb{E} \Big\| \sum_{|\alpha|\le m} x_\alpha w^\alpha  \Big\|^q\Big)^{1/q}
	\le C\sqrt{\frac{q}{2}}^{m}\Big(\mathbb{E} \Big\| \sum_{|\alpha|\le m}   \varepsilon_{i_\alpha}  x_\alpha  \Big\|^2\Big)^{1/2}.
\]
Joining this with \eqref{propcomparacion} and Jensen's inequality proves the second statement.
\end{proof}

Notice that for spaces $X$ such that $X$ and $X^*$ have GAP (see \cite[Theorem~1.7]{CaNi97} for a characterization) we can join Theorems \ref{propgap} and \ref{propgapduo} as well as both inequalities from Corollary \ref{corgap} (the finite cotype hypothesis is included in the fact that $X$ has GAP). Under such conditions we get estimates from above and below comparing the norm of a random polynomial with the norm of an independent random sum of its coefficients. In particular, this holds for Banach lattices of non-trivial type. Recall that $K$-convexity is the same as having non-trivial type and it is a self-dual property. So the dual of a Banach lattice $X$ of non-trivial type also has non-trivial type. Therefore, both $X$ and $X^*$ have finite cotype (see \cite[Corollary~11.9 and Proposition~11.10]{DiJaTo95}). Finally, since $X$ and $X^*$ are Banach lattices with finite cotype, both have GAP.

\section{One-variable decoupling}\label{sec1dec}
In this section we study one-variable decoupling of $m$-homogeneous tetrahedral polynomials. Instead of comparing $P(\xi)=M(\xi,\ldots,\xi)$ to its decoupled version $M(\xi^{(1)},\ldots,\xi^{(m)})$, we will only replace $\xi$ with an iid copy $\xi'$ in one entry to obtain $M(\xi',\xi,\ldots,\xi)$. Of course this time we are aiming at comparing the $p$-norms of both objects up to a constant $C$ not depending on $m$. Note that if the constants were allowed to grow with $m$, most of the results of this section would be direct applications of our previous results. To achieve our goal we have to use specific properties of the random variables at hand so we restrict ourselves to gaussian, Steinhaus and Rademacher variables.
We mention that some inequalities in this section can be thought of as Markov type inequalities for homogeneous polynomials in the sense of \cite{harris}.

We are mostly interested here in Steinhaus variables, but we start by dealing with gaussian variables to set a benchmark for the type of inequality we are looking for. We show a one-variable version of~\eqref{decogauss} that follows the ideas from \cite[Theorem~2]{Kw87}.

\begin{proposition}
\label{1decogauss}
Let $P:\C^n\rightarrow X$ be an $m$-homogeneous tetrahedral polynomial.
For every $1\le p<\infty$ we have
	\begin{align*}
	\frac{1}{\sqrt{e}}(\E\|P(\gamma)\|^p)^{1/p}
	\leq \sqrt m (\E\|M(\gamma',\gamma,\ldots,\gamma)\|^p)^{1/p}
	\leq \sqrt{e}(\E\|P(\gamma)\|^p)^{1/p}.
	\end{align*}
\end{proposition}
\begin{proof}
For the first inequality consider the gaussian variable
\[\gamma''=\frac{1}{\sqrt{m}}\gamma'+\sqrt{\frac{m-1}{m}}\gamma\sim \gamma.\] A straightforward calculation shows that for every $1\le i\le n$ we have
\[\E[\gamma_i'|\gamma'']=\frac{1}{\sqrt{m}}\gamma''_i \quad \text{and} \quad \E[\gamma_i|\gamma'']=\sqrt{\frac{m-1}{m}}\gamma''_i.\]
Therefore, since $P$ is tetrahedral we get
\[\E[M(\gamma',\gamma,\ldots,\gamma)|\gamma'']=\frac{1}{\sqrt{m}}\left(\sqrt{\frac{m-1}{m}}\right)^{m-1}P(\gamma'').\]
By Jensen's inequality we deduce
\begin{align*}
(\E\|P(\gamma'')\|^p)^{1/p}&\leq \Big(1+\frac{1}{m-1}\Big)^{(m-1)/2} \sqrt{m} (\E\|M(\gamma',\gamma,\ldots,\gamma)\|^p)^{1/p}
\\ &\leq \sqrt{em} (\E\|M(\gamma',\gamma,\ldots,\gamma)\|^p)^{1/p}.
\end{align*}
Since $\gamma''\sim\gamma$, this proves the first inequality.

For the second one, notice that for a Steinhaus variable $w$ we have
\[\E_w\Big[P\Big(\frac{w}{\sqrt{m-1}}\gamma'+\gamma\Big)\overline{w}\Big]=m M\Big(\frac{1}{\sqrt{m-1}}\gamma',\gamma,\ldots,\gamma\Big).\]
Again, by Jensen's inequality we deduce
\[\sqrt m (\E\|M(\gamma',\gamma,\ldots,\gamma)\|^p)^{1/p}
\leq \sqrt{\frac{m-1}{m}} \E_w \Big(\E_{\gamma,\gamma'}\Big\|P\Big(\frac{w}{\sqrt{m-1}}\gamma'+\gamma\Big)\Big\|^p\Big)^{1/p}.\]
Similarly as before, for a fixed $w\in\T$ we have that
\[\frac{w}{\sqrt{m-1}}\gamma'+\gamma\sim \sqrt{\frac{m}{m-1}}\gamma.\]
Therefore we obtain
\begin{align*}
\sqrt m (\E\|M(\gamma',\gamma,\ldots,\gamma)\|^p)^{1/p}
&\leq \sqrt{\frac{m-1}{m}}  \Big(\E_{\gamma}\Big\|P\Big(\sqrt{\frac{m}{m-1}}\gamma\Big)\Big\|^p\Big)^{1/p}
\\ &= \Big(1+\frac{1}{m-1}\Big)^{(m-1)/2}  (\E\|P(\gamma)\|^p)^{1/p}
\\ &\le \sqrt{e}(\E\|P(\gamma)\|^p)^{1/p}
\end{align*}
which completes the argument.
\end{proof}

Unlike in the previous section, translating the last result to other random variables is not possible since our comparison estimates (such as Theorem~\ref{compara}) involve constants of the form $C^m$ (and we aim at constants independent of $m$). However, we can recover an analogous decoupling inequality for Steinhaus polynomials assuming the Banach space is $K$-convex. Recall that a Banach space is $K$-convex if  the 1-homogeneous (Rademacher) projection $R$ is bounded on $L^2(\{-1,1\}^{\infty},X)$. We use that $K$-convexity is equivalent to the boundedness of the Rademacher projection $R$ on $L^p(\{-1,1\}^{\infty},X)$ for some (every) $1<p<\infty$ (see \cite[Lemma~7.4.3]{Hyto}). In this case, we denote the norm of $R$ by $K_p(X)$.

\begin{theorem}\label{1decoz}
Let $P:\C^n\rightarrow X$ be an $m$-homogeneous tetrahedral polynomial.
For every $1\leq p<\infty$ we have
	\begin{align}
	\label{ec1decoz}
	(\E\|P(w)\|^p)^{1/p}
	\leq \frac{\pi}{2} \sqrt{e m} (\E\|M(w',w,\ldots,w)\|^p)^{1/p}.
	\end{align}
If the Banach space $X$ is $K$-convex and $p>1$ we also have
\begin{align*}
\sqrt m (\E\|M(w',w,\ldots,w)\|^p)^{1/p}
	\leq \frac{\pi}{2} \sqrt{e} K_p(X)(\E\|P(w)\|^p)^{1/p}.
\end{align*}
\end{theorem}

\begin{remark}
The key fact in the proof of Proposition~\ref{1decogauss} is that for gaussian variables $a\gamma+b\gamma'\sim \sqrt{|a|^2+|b|^2}\gamma''$ for every $a,b\in \C$. This is no longer true for other random variables such as Steinhaus variables. To obtain a similar behavior we take advantage of the geometry in $\C$ and use the variables $w$ and $i\varepsilon w$ which always stay orthogonal to each other. Notice that in this case for $a,b\in \R$ we have that $aw+bi\varepsilon w\sim \sqrt{a^2+b^2}w$. This explains the $K$-convexity requirement which arises from estimating the norm of homogeneous projections involving Rademacher variables  we are forced to use in order to maintain the orthogonality between $w$ and $i\varepsilon w$. We do not know if the $K$-convextity assumption is necessary. It is worth mentioning that one can retrieve the full-decoupling inequality~\eqref{decoupsim} for Steinhaus variables by starting with the previous result and proceeding inductively. Therefore, weakening the $K$-convextity condition to assuming only finite cotype would provide a new proof of Corollary~\ref{decoup} for Steinhaus variables and possibly lead to better estimates.
\end{remark}

We begin by showing a slightly more general fact than~\eqref{ec1decoz} which will later lead to an interesting application.

\begin{lemma}\label{le1decoz}
Let $P:\C^n\rightarrow X$ be an $m$-homogeneous tetrahedral polynomial and fix $\lambda\in \C^n$.
For every $1\leq p<\infty$ we have
	\begin{align*}
	(\E_{w}\|M(\lambda w,w,\ldots,w)\|^p)^{1/p}
	\leq \sqrt{e m} (\E_{\varepsilon,w}\|M(\lambda\varepsilon w,w,\ldots,w)\|^p)^{1/p}.
	\end{align*}
\end{lemma}
\begin{proof}
Similarly as in Proposition~\ref{1decogauss}, for $1\le j\le n$ consider the random variables
\[w_j'=\frac{\sqrt{m-1}+i\varepsilon_j}{\sqrt{m}}w_j.\]
An easy calculation using rotation invariance shows that $w'\sim w$.
Notice that
\begin{align*}
\E[w_j|w']&=\E\Big[\frac{\sqrt{m}}{\sqrt{m-1}+i}w_j'\rchi_{\{\varepsilon_j=1\}}+\frac{\sqrt{m}}{\sqrt{m-1}-i}w_j'\rchi_{\{\varepsilon_j=-1\}} \Big|w'\Big]
\\ &=\frac{\sqrt{m}}{\sqrt{m-1}+i}w_j'\pp(\varepsilon_j=1|w')
+\frac{\sqrt{m}}{\sqrt{m-1}-i}w_j'\pp(\varepsilon_j=-1|w')
\\ &=\frac{1}{2} \Big(\frac{1}{\sqrt{m-1}+i}+
\frac{1}{\sqrt{m-1}-i}\Big)\sqrt{m}w_j'=\sqrt{\frac{m-1}{m}}w_j'.
\end{align*}
Therefore we also get
\begin{align*}
\E[i\varepsilon_j w_j|w']&=\E[\sqrt{m}w_j'-\sqrt{m-1}w_j |w']
=\sqrt{m}w_j'-\frac{m-1}{\sqrt{m}}w_j'=\frac{1}{\sqrt{m}}w_j'.
\end{align*}
Proceeding as in Proposition~\ref{1decogauss}, since $P$ is tetrahedral we get
\[\E[M(\lambda i\varepsilon w,w,\ldots,w)|w']=\frac{1}{\sqrt{m}}\sqrt{\frac{m-1}{m}}^{m-1}M(\lambda w',w',\ldots,w').\]
By Jensen's inequality we deduce
\begin{multline*}
(\E_w\|M(\lambda w,w,\ldots,w)\|^p)^{1/p}=(\E_{w'}\|M(\lambda w',w',\ldots,w')\|^p)^{1/p}
\\ \leq \Big(1+\frac{1}{m-1}\Big)^{(m-1)/2} \sqrt{m} (\E_{\varepsilon,w}\|M(\lambda i\varepsilon w,w,\ldots,w)\|^p)^{1/p}.
\end{multline*}
Finally, since $M$ is multilinear we can take $i$ out from the first coordinate which completes the proof.
\end{proof}

\begin{proof}[Proof of Theorem~\ref{1decoz}]
Applying the previous lemma for $\lambda =(1,\ldots,1)\in\C^n$ and~\eqref{zvse} we get
\begin{align*}
(\E_w\|P(w)\|^p)^{1/p}
	&\leq \sqrt{em} (\E_{\varepsilon,w}\| M(\varepsilon w,w,\ldots,w)\|^p)^{1/p}
	\\ &=\sqrt{em} \Big(\E_{\varepsilon,w}\Big\| \sum_{j=1}^n \varepsilon_j w_j M(e_j ,w,\ldots,w)\Big\|^p\Big)^{1/p}
	\\ &\le \frac{\pi}{2} \sqrt{e m} \Big(\E_{w,w'}\Big\| \sum_{j=1}^n w'_j w_j M(e_j ,w,\ldots,w)\Big\|^p\Big)^{1/p}
	\\ &= \frac{\pi}{2} \sqrt{e m}(\E_{w,w'}\| M(w'w,w,\ldots,w)\|^p)^{1/p}.
\end{align*}
By rotation invariance we can replace $w'w$ with $w'$ which yields~\eqref{ec1decoz}.

Now assume $X$ is $K$-convex and $p>1$. Using rotation invariance and~\eqref{zvse} in the other direction we get
\begin{align}\label{eq20}
(\E_{w,w'}\|M(w',w,\ldots,w)\|^p)^{1/p}
	\leq \frac{\pi}{2} (\E_{\varepsilon,w}\| M(\varepsilon w,w,\ldots,w)\|^p)^{1/p} .
\end{align}

As in the proof of the previous lemma, consider the Steinhaus variables
\[w_j'=\frac{\sqrt{m-1}+i\varepsilon_j}{\sqrt{m}}w_j.\]
For a fixed $z\in\T$ regard $P(\sqrt{m-1}z+i\varepsilon z)$ as a Walsh polynomial on $\varepsilon$. Notice that its $1$-homogeneous projection is $m M(i\varepsilon z,\sqrt{m-1}z,\ldots,\sqrt{m-1}z)$. Since $X$ is $K$-convex we deduce
\begin{multline*}
(\E_{\varepsilon,w}\|m M(i\varepsilon w,\sqrt{m-1}w,\ldots,\sqrt{m-1}w)\|^p)^{1/p}\le
\\\le K_p(X) (\E\|P(\sqrt{m-1}w+i\varepsilon w)\|^p)^{1/p} \notag
= \sqrt{m}^m K_p(X)(\E_{w'}\|P(w')\|^p)^{1/p}.
\end{multline*}
So we get
\begin{align}\label{eq21}
\sqrt{m} (\E_{\varepsilon,w}&\| M(\varepsilon w,w,\ldots,w)\|^p)^{1/p}= \notag
	\\ &= \frac{1}{\sqrt{m}\sqrt{m-1}^{m-1}}  (\E_{\varepsilon,w}\|m M(i\varepsilon w,\sqrt{m-1}w,\ldots,\sqrt{m-1}w)\|^p)^{1/p}\notag
	\\ &\le  \Big(1+\frac{1}{m-1}\Big)^{(m-1)/2} K_p(X)
	(\E_{w'}\|P(w')\|^p)^{1/p}\notag
	\\ &=  \sqrt{e}K_p(X)
	(\E_{w}\|P(w)\|^p)^{1/p}.
\end{align}
Joining~\eqref{eq20} and~\eqref{eq21} concludes the argument.
\end{proof}

As an application of the previous results we have the following corollary (see \cite{harris,DeMa20} for similar inequalities and their applications).

\begin{corollary}\label{cor1deco}
Let $X$ be a $K$-convex space, $P:\C^n\rightarrow X$ an $m$-homogeneous tetrahedral polynomial. For every $1\leq p<\infty$ and every $\lambda \in \C^n$ we have
\begin{align*}
(\E\|\langle \nabla P(w),\lambda w\rangle\|^p)^{1/p}
\leq \frac{\pi e}{2} K_p(X)m\|\lambda \|_\infty(\E\|P(w)\|^p)^{1/p}.
\end{align*}
\end{corollary}
\begin{proof}
Notice that if $P(z)=\sum_A x_A z_A$ we can rewrite
\[\langle \nabla P(z),\lambda z\rangle= \sum_{j=1}^n \lambda_j z_j \sum_{A / j\in A} x_A z_{A-\{j\}}
=\sum_{A} \sum_{j\in A} \lambda_j   x_A z_{A}=m M(\lambda z,z,\ldots,z).\]
Applying Lemma~\eqref{le1decoz} and~\eqref{zvse} we get
\begin{align*}
(\E\|\langle \nabla P(w),\lambda w\rangle\|^p)^{1/p}
&\le \sqrt{e}m^{3/2}(\E_{\varepsilon,w}\|M(\lambda \varepsilon w,w,\ldots,w)\|^p)^{1/p}
\\ &\le \frac{\pi}{2}\sqrt{e} m^{3/2}\|\lambda \|_\infty(\E_{\varepsilon,w}\|M(\varepsilon w,w,\ldots,w)\|^p)^{1/p}.
\end{align*}
The result follows using~\eqref{eq21} from the previous theorem.
\end{proof}

To illustrate how we may use the last corollary, suppose we want to estimate $\E_w\|P(e^{i\theta} w)-P(w)\|^p$ for a small perturbation $e^{i\theta}=(e^{i\theta_j})_{j} \in\T^n$ where every coordinate is close to $1$. For $z\in\C^n$, consider the function $f_z(t)=P(e^{i\theta t}z)$. Observe that
\[f_z'(t)=\langle \nabla P(e^{i\theta t}z),i\theta e^{i\theta t}z\rangle.\]
So under the assumptions of Corollary~\ref{cor1deco} we get
\begin{multline*}
(\E_w\|P(e^{i\theta} w)-P(w)\|^p)^{1/p} = (\E_w\|f_w(1)-f_w(0)\|^p)^{1/p}
\leq (\E_w\|f'_w(t_0)\|^p)^{1/p}
\\ = (\E\|\langle \nabla P(w),i \theta w\rangle\|^p)^{1/p}
\le \frac{\pi e}{2} K_p(X)m\|\theta\|_\infty(\E\|P(w)\|^p)^{1/p}.
\end{multline*}
Of course, this estimate is only useful when $\|\theta\|_\infty$ is smaller than $1/m$. Otherwise the triangle inequality together with rotation invariance already ensures that
\[(\E_w\|P(e^{i\theta} w)-P(w)\|^p)^{1/p} \le 2(\E\|P(w)\|^p)^{1/p}.\]

Unfortunately, we do not know sharp one-variable decoupling estimates for other random variables. We finish this note by discussing one-variable decoupling for Walsh polynomials and provide some partial results.

A famous inequality of Pisier (see \cite[Lemma~7.3]{pisier}) states that for every $f:\{-1,1\}^n\rightarrow X$ and $1\le p <\infty$ we have
\begin{align}\label{pisin}
(\E_\varepsilon \|f(\varepsilon)-\E f\|^p)^{1/p}
\le 2e \log n (\E_{\varepsilon,\varepsilon'} \|\langle \nabla f(\varepsilon),\varepsilon'\rangle\|^p)^{1/p}.
\end{align}
This was used by Pisier to study a non-linear version of type for metric spaces known as Enflo type. He proved that for Banach spaces both notions almost coincide. More precisely, it is easy to check that Enflo type $p$ implies type $p$ for Banach spaces (see for example \cite{introribe}). Conversely, in \cite[Theorem~7.5]{pisier} it is shown that Banach spaces with type $p>1$ enjoy Enflo type $r$ for every $1\leq r < p$. The $\log n$ term in~\eqref{pisin} was the reason why one could only deduce Enflo type $r$ instead of Enflo type $p$. So the question of whether the $\log n$ factor could be removed for Banach spaces of non-trivial type ($K$-convex spaces) became a long-standing open problem. Quite recently in \cite{paa} it was proven that this is true even for spaces with finite cotype and the coincidence of type and Enflo type for Banach spaces was settled.

In the case of $m$-homogenous polynomials we show that for spaces $X$ of cotype $q<\infty$ and for $1\le p\le q$ we have
\begin{align*}
(\E \|P(\varepsilon)\|^p)^{1/p}
\le C m^{-1/q} (\E \|\langle \nabla P(\varepsilon),\varepsilon'\rangle\|^p)^{1/p}.
\end{align*}
Notice that since $P$ is $m$-homogenous we have $\E P=0$ so this inequality is a variant of~\eqref{pisin} for $m$-homogeneous polynomials. Also a straightforward computation shows that
\begin{align}
\label{ecgrad}
\langle\nabla P(\varepsilon),\varepsilon'\rangle =m M(\varepsilon',\varepsilon,\ldots,\varepsilon).
\end{align}
So in our notation the previous inequality can be stated as follows.

\begin{proposition}\label{casi}
 Let $X$ be a Banach space of finite cotype $q$ and fix $1\le p\le q$. There is a constant $C\ge 1$ such that for every $m$-homogenous polynomial $P:\{1,1\}^n\rightarrow X$ we have
\begin{align*}
(\E \|P(\varepsilon)\|^p)^{1/p}
\le C m^{1-1/q} (\E \|M(\varepsilon',\varepsilon,\ldots,\varepsilon)\|^p)^{1/p}.
\end{align*}
\end{proposition}

Regrettably, this inequality only provides an analogue to~\eqref{ec1decoz} for cotype 2 spaces where $m^{1-1/q}$ becomes $\sqrt{m}$. It should be mentioned that for $X=L^q(\mu)$ and $p=q$ it is easily proven that the above result holds replacing $m^{1-1/q}$ with $\sqrt{m}$. This suggests that $\sqrt{m}$ could be the right bound regardless of the cotype of $X$.

In order to prove Proposition~\ref{casi} one can replace the integral from 0 to $\infty$ in \cite[Theorem~1.4]{paa} with an integral from 0 to $1/m$ and carefully follow the proof of this result as well as \cite[Proposition~4.2]{paa}. However, we take a more direct route using a combinatorial identity from Rzeszut and Wojciechowski which can be retrieved from equations (3.12) through (3.15) from \cite{RzWo_19}. This identity inspired our Lemma~\ref{lemidcomb2} and can be regarded as a one-variable decoupling version of it. Given a vector space $V$, a family $\{ v_{A} \colon A \subseteq [n], \, \vert A \vert =m \}\subseteq V$ (where $n,m \in \mathbb{N}$) and $k \in \mathbb{N}$ we have
\begin{equation}\label{idcomb}
\begin{split}
	\sum_{\substack{B\subseteq[n] \\ |B|=k}}
	&\sum_{\substack{A_1\subseteq B \\ |A_1|=1}}
	\sum_{\substack{A_2\subseteq B^c \\ |A_2|=m-1}} v_{A_1\cup A_2}
	=
	\sum_{\substack{B\subseteq[n] \\ |B|=k}}
	\sum_{\substack{A\subseteq [n] \\ |A|=m \\|A\cap B|=1}} v_A
	=
	\sum_{\substack{A\subseteq[n] \\ |A|=m}}
	\sum_{\substack{B\subseteq [n] \\ |B|=k \\ |A\cap B|=1}} v_A
	\\ &=
	\sum_{\substack{A\subseteq[n] \\ |A|=m}}
	\big|\{B\subseteq [n] : \ |B|=k, \ |A\cap B|=1\}\big| v_A
	=
	m \binom{n-m}{k-1} \sum_{\substack{A\subseteq [n] \\ |A|=m}} v_A.
\end{split}
\end{equation}
Notice that if $n=km$, then $\binom{n-m}{k-1} = \binom{(k-1)m}{k-1}$.
	Analogously, as in the decomposition from Section~\ref{secdeco}, a straightforward (but tedious) application of Stirling's formula yields
\[
	\frac{\binom{n}{k}}{m \binom{n-m}{k-1}} \leq 4.
\]
The main idea is that in this case we can compare
\[\sum_{\substack{A\subseteq [n] \\ |A|=m}} v_A \quad \text{vs.}\quad
\frac{1}{\binom{n}{k}}\sum_{\substack{B\subseteq[n] \\ |B|=k}}
	\sum_{\substack{A_1\subseteq B \\ |A_1|=1}}
	\sum_{\substack{A_2\subseteq B^c \\ |A_2|=m-1}} v_{A_1\cup A_2}.\]
	Whereas the left-hand side will be the object of study, the right-hand side is an average (over all subsets $B$) of an expression with a decoupled structure. This is because for a fixed $B$ the indices $A$ have been split in $A_1$ and $A_2$ in such a way that $A_1$ has always 1 element, $A_2$ has always $m-1$ elements and these elements do not mix since $A_1\subseteq B$ and $A_2\subseteq B^c$.
	
\begin{proof}[Proof of Proposition~\ref{casi}]
Let $\{ x_{A} \colon A \subseteq [n], \, \vert A \vert =m \}$ be a family of vectors in $X$. Taking $n$ larger if necessary, we may assume $n=km$ for some $k\in \N$ so we have
\[\frac{1}{m \binom{n-m}{k-1}} \leq \frac{4}{\binom{n}{k}}.
\]
Using~\eqref{idcomb} for $v_A=x_A \varepsilon_A$ we get
\begin{align}
\label{ecic2}
\Big(\mathbb{E}_\varepsilon\Big\|\sum_{|A|=m} x_A \varepsilon_A\Big\|^p\Big)^{1/p}
&= \frac{1}{m \binom{n-m}{k-1}}\Big(\mathbb{E}_\varepsilon\Big\| \sum_{\substack{B\subseteq[n] \\ |B|=k}} \sum_{\substack{A_1\subseteq B \notag \\ |A_1|=1}}
	\sum_{\substack{A_2\subseteq B^c \\ |A_2|=m-1}} x_{A_1\cup A_2}\varepsilon_{A_1\cup A_2}\Big\|^p\Big)^{1/p} \notag
\\&\leq \frac{4}{\binom{n}{k}}\sum_{\substack{B\subseteq[n] \\ |B|=k}}\Big(\mathbb{E}_\varepsilon\Big\|\sum_{\substack{A_1\subseteq B \\ |A_1|=1}}
	\varepsilon_{A_1} \sum_{\substack{A_2\subseteq B^c \\ |A_2|=m-1}} x_{A_1\cup A_2}\varepsilon_{A_2}\Big\|^p\Big)^{1/p}.
\end{align}
Now since $A_1\subseteq B$ and $A_2\subseteq B^c$ these indices never overlap. So for a fixed $B$ we get
\begin{align}
\label{ecic3}
\mathbb{E}_\varepsilon\Big\|\sum_{\substack{A_1\subseteq B \\ |A_1|=1}}
	\varepsilon_{A_1} \sum_{\substack{A_2\subseteq B^c \\ |A_2|=m-1}} x_{A_1\cup A_2}\varepsilon_{A_2}\Big\|^p
	=\mathbb{E}_{\varepsilon,\varepsilon'}\Big\|\sum_{\substack{A_1\subseteq B \\ |A_1|=1}}
	\varepsilon'_{A_1} \sum_{\substack{A_2\subseteq B^c \\ |A_2|=m-1}} x_{A_1\cup A_2}\varepsilon_{A_2}\Big\|^p.
	\end{align}
Denote $\partial_j P = \frac{\partial P}{\partial z_j}$ and notice that
\[\partial_j P (\varepsilon)= \sum_{\substack{A\subseteq [n]-\{j\} \\ |A|=m-1}} x_{\{j\}\cup A} \varepsilon_A.\]
We have
	\begin{align}
	\label{ecic4}
	\E_{\varepsilon_i/ i\in B}\Big[ \sum_{j\in B} \varepsilon'_j \partial_j P(\varepsilon)\Big]
	&=\E_{\varepsilon_i/ i\in B}\Big[\sum_{\substack{A_1\subseteq B \\ |A_1|=1} }
	\varepsilon'_{A_1} \sum_{\substack{A_2 \subseteq A_1^c \\ |A_2|=m-1}} x_{A_1\cup A_2}\varepsilon_{A_2}\Big] \notag
	\\ &=\sum_{\substack{A_1\subseteq B \\ |A_1|=1}}
	\varepsilon'_{A_1} \sum_{\substack{A_2\subseteq B^c \\ |A_2|=m-1}} x_{A_1\cup A_2}\varepsilon_{A_2}.
	\end{align}
Joining~\eqref{ecic2},~\eqref{ecic3},~\eqref{ecic4} and using Jensen's inequality we obtain
\begin{align}
\label{ecic5}
	\Big(\mathbb{E}_\varepsilon\Big\|\sum_{|A|=m} x_A \varepsilon_A\Big\|^p\Big)^{1/p}
	\leq \frac{4}{\binom{n}{k}}\sum_{\substack{B\subseteq[n] \\ |B|=k}}\Big(\mathbb{E}_{\varepsilon,\varepsilon'}\Big\| \sum_{j\in B} \varepsilon'_j \partial_j P(\varepsilon)\Big\|^p\Big)^{1/p}.
\end{align}
Recall that $n=km$ and the definition of $\Pi_{k,m}$ from~\eqref{party} as the familiy of all ordered partitions $\pi=(B_1,\ldots,B_m)$ of $[n]$ in $m$ sets $B_l$ of $k$-elements.
By a symmetry argument observe that
\begin{align}
\label{parti}
\frac{1}{\binom{n}{k}}\sum_{\substack{B\subseteq[n] \\ |B|=k}}\Big(\mathbb{E}_{\varepsilon,\varepsilon'}\Big\| \sum_{j\in B} \varepsilon'_j \partial_j P(\varepsilon)\Big\|^p\Big)^{1/p}
&=\frac{1}{m|\Pi_{k,m}|}\sum_{\pi\in\Pi_{k,m}}\sum_{l=1}^m\Big(\mathbb{E}_{\varepsilon,\varepsilon'}\Big\| \sum_{j\in B_l} \varepsilon'_j \partial_j P(\varepsilon)\Big\|^p\Big)^{1/p}.
\end{align}
Using H\"older's and Minkowski's inequalities we get
\begin{align}\label{eqinter}
\sum_{l=1}^m\Big(\mathbb{E}_{\varepsilon,\varepsilon'}\Big\| \sum_{j\in B_l} \varepsilon'_j \partial_j P(\varepsilon)\Big\|^p\Big)^{1/p}
&\leq m^{1/q'} \Big(\sum_{l=1}^m\Big(\mathbb{E}_{\varepsilon,\varepsilon'}\Big\| \sum_{j\in B_l} \varepsilon'_j \partial_j P(\varepsilon)\Big\|^p\Big)^{q/p} \Big)^{1/q} \notag
\\ & \leq m^{1/q'} \Big(\mathbb{E}_{\varepsilon,\varepsilon'}\Big(\sum_{l=1}^m\Big\| \sum_{j\in B_l} \varepsilon'_j \partial_j P(\varepsilon)\Big\|^q\Big)^{p/q} \Big)^{1/p}
\end{align}
Applying the cotype $q$ inequality and the Kahane-Khinchin inequality to change the exponent $q$ to $p$ (see \cite[11.1]{DiJaTo95}), we have
\begin{align*}
\Big(\sum_{l=1}^m\Big\| \sum_{j\in B_l} \varepsilon'_j \partial_j P(\varepsilon)\Big\|^q\Big)^{1/q}
&\le C \Big(\mathbb{E}_{\delta} \Big\| \sum_{l=1}^m \delta_l \sum_{j\in B_l} \varepsilon'_j \partial_j P(\varepsilon)\Big\|^q\Big)^{1/q}
\\ &\le \widetilde{C} \Big(\mathbb{E}_{\delta} \Big\| \sum_{l=1}^m \delta_l \sum_{j\in B_l} \varepsilon'_j \partial_j P(\varepsilon)\Big\|^p\Big)^{1/p}.
\end{align*}
Joining this with \eqref{eqinter} we obtain
\begin{align*}
\sum_{l=1}^m\Big(\mathbb{E}_{\varepsilon,\varepsilon'}\Big\| \sum_{j\in B_l} \varepsilon'_j \partial_j P(\varepsilon)\Big\|^p\Big)^{1/p}
&\leq m^{1/q'} \widetilde{C} \Big(\mathbb{E}_{\varepsilon,\varepsilon',\delta} \Big\| \sum_{l=1}^m \delta_l \sum_{j\in B_l} \varepsilon'_j \partial_j P(\varepsilon)\Big\|^p\Big)^{1/p}
\\ & = m^{1/q'} \widetilde{C} \Big(\mathbb{E}_{\varepsilon,\varepsilon'} \Big\| \langle \nabla P(\varepsilon) , \varepsilon'\rangle\Big\|^p\Big)^{1/p},
\end{align*}
where in the last step we used that $\delta_l\varepsilon'_j\sim \varepsilon'_j$. Combining this with~\eqref{ecic5} and~\eqref{parti} we have
\[\Big(\mathbb{E}_\varepsilon\Big\|\sum_{|A|=m} x_A \varepsilon_A\Big\|^p\Big)^{1/p}
	\leq  4 \widetilde{C} m^{-1/q} \Big(\mathbb{E}_{\varepsilon,\varepsilon'} \Big\| \langle \nabla P(\varepsilon) , \varepsilon'\rangle\Big\|^p\Big)^{1/p}.\]
	The result follows from~\eqref{ecgrad}.
	\end{proof}

\bibliographystyle{abbrv}
\bibliography{biblio}

\end{document}